\documentclass[11pt,english]{article}
\usepackage[margin = 2.5cm]{geometry}

\usepackage{amsmath}
\usepackage{amsthm}
\usepackage{amssymb}
\usepackage{mathtools}
\usepackage{graphicx}
\usepackage{xcolor}
\usepackage[hidelinks, bookmarks = false]{hyperref}
\usepackage{enumitem}
\usepackage{bbm}

\usepackage{caption}
\newtheorem{theorem}{Theorem}[section]
\newtheorem{corollary}[theorem]{Corollary}

\newtheorem{lemma}[theorem]{Lemma}

\newtheorem{definition}[theorem]{Definition}

\newtheorem{conjecture}[theorem]{Conjecture}

\newcounter{claimcounter}[theorem]
\newtheorem{claim}[claimcounter]{Claim}

\def\b1K{\mbox{\boldmath $K$}_{-1}}
\def\bK{\mbox{\boldmath $K$}}

\newbox\noforkbox \newdimen\forklinewidth
\forklinewidth=0.3pt \setbox0\hbox{$\textstyle\smile$}
\setbox1\hbox to \wd0{\hfil\vrule width \forklinewidth depth-2pt
 height 10pt \hfil}
\wd1=0 cm \setbox\noforkbox\hbox{\lower 2pt\box1\lower
2pt\box0\relax}

\def\sub'm{\prec_{\bK'}}
\def\grpf #1 #2{{\rm grp}_{#2}(#1)}
\def\fldf #1 #2{{\rm fld}_{#2}(#1)}
\def\dclf #1 #2{{\rm dcl}_{#2}(#1)}
\def\rclf #1 #2{{\rm rcl}_{#2}(#1)}
\def\aclf #1 #2{{\rm acl}_{#2}(#1)}
\def\acff #1 #2{{\rm acf}_{#2}(#1)}
\def\strf #1 #2{{\rm str}_{#2}(#1)}
\def\tclf #1 #2{{\rm acf}_{#2}(#1)}

\def\hbar{{\bf h}}

\date{\today}

\newcommand{\F}{\mathcal{F}}
\newcommand{\M}{\mathcal{M}}
\newcommand{\A}{\mathcal{A}}

\newcommand{\G}{\mathcal{G}}

\newcommand{\D}{\mathcal{D}}

\newcommand{\J}{\mathcal{J}}
\newcommand{\N}{\mathcal{N}}
 
\newcommand{\cP}{\mathcal{P}}

\newcommand{\Z}{\mathcal{Z}}

\newcommand{\cH}{\mathcal{H}}
\newcommand{\cL}{\mathcal{L}}
\newcommand{\Pc}{\mathcal{P}}

\newcommand{\ve}{\varepsilon}

\newcommand{\ex}{\mathrm{ex}}

\usepackage{mathtools}

\newcommand{\floor}[1]{\left\lfloor #1 \right\rfloor}
\newcommand {\tuple}[1]{\langle #1 \rangle}

\newtheorem{prob}{Problem}

\usepackage{xpatch}
\xpatchcmd{\proof}{\itshape}{\normalfont\proofnamefont}{}{}
\newcommand{\proofnamefont}{}
\renewcommand{\proofnamefont}{\bfseries}

\title{Rational exponents near $3/2$}
 \author{Tao Jiang \footnote{Dept. of Mathematics, Miami University, Oxford, OH 45056, USA, {\tt jiangt@miamioh.edu}.  }\and Sean Longbrake \footnote{Dept. of Mathematics, Emory University,  Atlanta, GA 30322, USA, {\tt sean.longbrake@emory.edu}} \and Liana Yepremyan \footnote{Dept. of Mathematics, Emory University,  Atlanta, GA 30322, USA {\tt lyeprem@emory.edu}.  Research is supported by the National Science Foundation grant 2247013: Forbidden and Colored Subgraphs.}}

\begin{document}

\maketitle

\begin{abstract}
Given a graph $H$, the extremal number $\ex(n,H)$ is the maximum number of edges in an $n$-vertex
graph not containing $H$ as a subgraph. The well-known rational exponents conjecture of Erd\H{o}s and Simonovits
states that for any rational $\gamma\in (1,2)$ there exists a single bipartite graph $H$ satisfying $\ex(n,H)=\Theta(n^\gamma)$.
Among other results, the conjecture has been verified for all $\gamma=1+a/b$, where $b>a^2$, by Jiang and Qiu \cite{JQ} and
for all $\gamma=2-a/b$, where $b>\max\{a, (a-1)^2\}$, by Conlon and Janzer \cite{CJ}. 

In this paper, we establish the rational exponents conjecture for many $\gamma$ near the center of the interval, namely, for all $\gamma=1+\frac{rt-1}{2rt+2r}$, where $r,t$ are natural numbers satisfying $t\geq 2$,  $r\geq 2t+3$.
\end{abstract}

\section{Introduction}

Given a family $\cH$ of graphs, a graph $G$ is {\it $\cH$-free} if $G$ does not contain any member of $\cH$ as a subgraph. The {\it extremal number} $\ex(n,\cH)$ of $\cH$ (also called the {\it Tur\'an number} of $\cH$) is the largest number of edges in an $\cH$-free $n$-vertex graph.
When $\cH$ consists of a single graph $H$, we write $\ex(n,H)$ for $\ex(n,\cH)$. The study of the function $\ex(n,\cH)$ is usually referred
to as the {\it Tur\'an problem}, and it is one of the central problems in extremal graph theory. When $\cH$ consists only of non-bipartite graphs, the celebrated
Erd\H{o}s-Stone-Simonovits theorem \cite{ES, E-Stone} determines the asymptotics of $\ex(n,\cH)$. When $\cH$ contains a bipartite graph,
much less is known in general. Nevertheless, in recent years, much progress has been made on the Tur\'an problem for bipartite graphs,
particularly on the so-called {\it rational exponents conjecture} of Erd\H{o}s and Simonovits (see, for example \cite{Erdos-problem}).

\begin{conjecture}[Rational exponents conjecture] \label{conj:ES}
For every rational number $\gamma\in [1,2]$, there exists a graph $H$ with $\ex(n,H)=\Theta(n^\gamma)$.
\end{conjecture}

The biggest breakthrough on the conjecture was made by Bukh and Conlon \cite{BC} who showed that for any rational number $\gamma\in [1,2]$
there exists a finite family $\cH$ of graphs such that $\ex(n,\cH)=\Theta(n^\gamma)$. Following this breakthrough, much progress has been made on
the conjecture in its original form (which asks for a single graph $H$ rather than a family $\cH$) by various authors (see \cite{CJL, Janzer-bipartite, JJM, JMY, JQ, KKL} for instance) and Conjecture \ref{conj:ES} has now been verified for many values of $\gamma$. In particular, a theorem of Jiang and Qiu \cite{JQ} established many exponents near one while a theorem of 
of Conlon and Janzer \cite{CJ} established many exponents near two.

\begin{theorem} [\cite{JQ}] \label{thm:JQ}
For each rational of the form $\gamma=1+a/b$, where $a,b$ are natural numbers with $b>a^2$, there exists a bipartite
graph $H$ with $\ex(n,H)=\Theta(n^\gamma)$.
\end{theorem}

\begin{theorem} [\cite{CJ}] \label{thm:CJ}
For each rational number of the form $\gamma=2-a/b$, where $a,b$ are natural numbers with $b\geq \{a, (a-1)^2\}$,
there exists a bipartite graph $H$ with $\ex(n,H)=\Theta(n^\gamma)$.
\end{theorem}

In this paper, we are interested in establishing many exponents in the middle of the spectrum, near $3/2$.
Given a graph $H$, let $H'$ denote the graph obtained from $H$ by subdividing each edge of $H$ exactly once.
We note that a theorem of Janzer \cite{Janzer-bipartite} shows $\ex(n,K'_{s,t})=\Theta(n^{\frac{3}{2}-\frac{1}{2s}})$ when $t$ is sufficiently large relative to $s$, which gives infinitely one-parameter family of  exponents near $3/2$.
Our main result establishes an infinite two-parameter family of exponents near $3/2$.

 To describe our result, we need some definitions,
first introduced in \cite{BC}. Let $F$ be a graph and $R$ a proper subset of $V(F)$, called the {\it root set}.
For each nonempty subset $S\subseteq V(F)$, let $e_S$ denote the number of edges of $F$ incident to a vertex in $S$ and let $\rho_F(S)=\frac{e_S}{|S|}$.
Let $\rho(F)=\rho_F(V(F)\setminus R)$. We say that $(F,R)$ (or $F$ if $R$ is clear) is {\it balanced} if $\rho_F(S)\geq \rho(F)$ holds for each nonempty $S\subseteq V(F)\setminus R$. 

Let $F_R^\ell$ denote the graph consisting of $\ell$ labeled copies $F_1,\dots, F_\ell$ of $F$ that
such that for each $v\in R$, the images of $v$ in $F_1,\dots F_\ell$ are the same
and that $F_1,\dots, F_\ell$ are pairwise vertex disjoint outside the common image of $R$.
We call $F_R^\ell$ the {\it $\ell$-th power} of $F$ {\it rooted at $R$}.
When the context is clear, we will drop the subscript $R$. 

Bukh and Conlon \cite{BC} proved the following celebrated result.
\begin{theorem} [\cite{BC}] \label{lem:BC-lower}
For every balanced rooted bipartite graph $(F,R)$ with $\rho(F)>0$, there exists 
a positive integer $\ell_0=\ell_0(F)$ such that for all $\ell\geq \ell_0$, $\ex(n,F_R^\ell)=\Omega(n^{2-\frac{1}{\rho(F)}})$.
\end{theorem}

Motivated by Theorem~\ref{lem:BC-lower}, Bukh and Conlon proposed the following conjecture as a way of
tackling the rational exponents conjecture.

\begin{conjecture}[Bukh-Conlon Conjecture \cite{BC}] \label{conj:BC}
For any balanced rooted tree $(T,R)$ and any natural number $\ell$, we have
\[\ex(n, T^\ell_R)=O_\ell(n^{2-\frac{1}{\rho(T)}}).\]
\end{conjecture}

Indeed, for any balanced rooted tree $(T,R)$ for which Conjecture~\ref{conj:BC} holds, by Theorem~\ref{lem:BC-lower},
$\ex(n,T^\ell_R)=\Theta(n^{2-\frac{1}{\rho(T)}})$ holds for sufficiently large $\ell$. Since
for any rational $\gamma\in (1,2)$, it is easy to construct a balanced rooted tree $(T,R)$ (see \cite{BC}) with $\rho(T)=\frac{1}{2-\gamma}$, Conjecture~\ref{conj:BC} implies Conjecture~\ref{conj:ES}.

Let $r,t$ be positive integers. Let $T_{r,t}$ denote the height two tree obtained from an $r$-star by joining $t$ leaves to each
leaf of it.  Let $R$ be the set of leaves of $T_{r,t}$. For each positive integer $\ell$, let $F_{r,t}^\ell=[T_{r,t}]^\ell_R$.
The following theorem was proved Conlon and Janzer \cite{CJ} (see \cite{JJM} for an earlier and weaker version of the theorem), which forms the most important ingredient of Theorem \ref{thm:CJ}. 

\begin{theorem} \label{thm:height-two}
Let $\ell,r,t$ be positive integers, where $r\geq t+2\geq 3$. Then $\ex(n,F^\ell_{r,t})=O(n^{2-\frac{r+1}{rt+r}})$.
\end{theorem}

Recall that $T'_{r,t}$ is the tree obtained from $T_{r,t}$ by subdividing each edge of $T_{r,t}$ once.
Let $R$ denote the set of leaves of $T'_{r,t}$.  For each positive integer $\ell$, let $H_{r,t}^\ell=[T'_{r,t}]^\ell_R$.
Equivalently, $H^\ell_{r,t}=[F^\ell_{r,t}]'$. Our main result is the following.
\begin{theorem} [Main theorem] \label{thm:main} 
Let $\ell,r,t$ be positive integers, where $t \geq 2, r \geq 2t + 3$. Then, 
\[\ex(n,H^\ell_{r,t})=O(n^{1+\frac{rt-1}{2rt+2r}}).\]
\end{theorem}

Note that the $t=1$ case was solved as a special case of the main theorem of \cite{Janzer-bipartite}.
Let us provide some contexts to our main result. 
First, it gives the first large two-parameter infinitely family of rational exponents
near $3/2$, complementing the earlier large family near $1$ and large family near $2$.
Second, it verifies the Bukh-Conlon conjecture for a new large family of trees, namely, subdivisions
of balanced height two tree $T_{r,t}$ with $t \geq 2, r\geq 2t  + 3$. Previously, to the best of our knowledge, 
the Bukh-Conlon conjecture has only been verified for height two trees $T_{r,t}$,
for $r\geq t^3-1$ in \cite{JJM} and for $r\geq t+2\geq 3$
in \cite{CJ}, for certain spiders (i.e subdivisions of stars)
in \cite{CJL}, \cite{Janzer-bipartite}, \cite{JQ}, \cite{KKL}, for balanced double stars and
for $3$-combs in \cite{JMY}, and for $4$-combs in \cite{KKL}, where a {\it $t$-comb} is the tree obtained
by adding a leaf to each vertex of a $t$-vertex path. 
The proof of our main theorem is intricate in nature.
We substantially expand on the existing tools while adding several novel ideas.
First, we expand on the ideas in \cite{BC, JJM} to develop a template that makes extensive and versatile
use of so-called {\it anchored-subfamilies}, which we feel can potentially be further developed into
a systematical approach for establishing the Bukh-Conlon conjecture for more families of rooted trees.
Second, we develop a few embedding lemmas for subdivisions, which combine dependent random choice approach
with anchored subfamily notion and may be viewed as
a dependent random choice type template for subdivisions. We believe these lemmas, particularly developed
for the asymmetric bipartite setting, will be useful for future work on related Tur\'an problems for bipartite graphs.

We organize the rest of the paper as follows. In Section~\ref{sec:notation} we introduce some notation
and prove a few lemmas. In Section~\ref{sec:strong-anchor}, we introduce two crucial notions and prove
a few properties of these two notions. In Section~\ref{sec:subdiv-embed}, we establish a few useful
embedding lemmas for subdivisions. In Section~\ref{sec:main-proof}, we give the proof of the main theorem,
Theorem~\ref{thm:main}.


\section{Notation and bounds on some substructures} \label{sec:notation}

Given a graph $G$, we use $\Delta(G), \delta(G)$ to denote the maximum and minimum degree of $G$, respectively.
Let $K\geq 1$ be a real. We say that a graph $G$ is {\it $K$-almost regular} if $\Delta(G)\leq K \delta(G)$.
Given a vertex $v$ in $G$, let $N_G(v)$ denote the {\it neighborhood} of $v$ in $G$. 
Given a set $S$ of vertices, let $N^*_G(S)=\bigcap_{v\in S} N_G(v)$ denote the {\it common neighborhood} of $S$ in $G$.
When the context is clear, we drop the subscript.
Below, we formally define the trees $T_{r,t}$ and $T'_{r,t}$.

\begin{definition} [$T_{r,t}$ and its subdivision] \label{def:Trt}
{\rm We define the tree $T_{r,t}$ to be the graph with vertex set 
$V=\{w\}\cup\{y_i: \in [r]\}\cup \{z_{i,j}: i\in [r], j\in [t]\}$
and edge set $E=\{wy_i: i\in [r]\}\cup \bigcup_{i=1}^r \{y_iz_{i,j}: j\in [t]\}$. We call $w$ the \emph{center} of $T_{r,t}$. We let $T'_{r,t}$ denote the subdivision of $T_{r,t}$.
For each $i\in [r]$, let $x_i$ denote the vertex used to subdivide the edge $wy_i$.
For each $i\in [r],  j\in [t]$, let $v_{i,j}$ denote the vertex used to subdivde the edge 
$y_i z_{i,j}$. Hence, $E(T'_{r,t})=\{wx_i: i\in [r]\}\cup \{x_iy_i: i\in [i]\}\cup\{y_iv_{ij}: i\in [r], j\in [t]\}\cup \{v_{i,j}z_{i,j}: i\in [r], j\in [t]\}$.
}
\end{definition}

In this paper, we  use a standard tool to reduce a dense host graph to an almost regular induced subgraph with similar relative density. 
Such a reduction tool was first developed by Erd\H{o}s and Simonovits \cite{ES-regular}.
For convenience, we will use the following version.

\begin{lemma}[\cite{Jiang-Seiver},  Proposition 2.7] \label{lem:almost-regular}
Let $\varepsilon, C$ be positive reals where $\varepsilon <1$ and $C\geq 1$. Let $n$ be a positive integer that is sufficiently large as a function of $\varepsilon$. Let $G$ be an $n$-vertex graph with $e(G)\geq Cn^{1+\varepsilon}$. Then $G$ contains a subgraph $G$ on $m\geq n^{\frac{\varepsilon(1-\varepsilon)}{2(1+\varepsilon)}}$ such that $e(G')\geq \frac{2C}{5} m^{1+\varepsilon}$ and $\Delta(G')\geq K\delta(G')$, where $K=20\cdot 2^{\frac{1}{\varepsilon^2+1}}$.
\end{lemma}
Throughout the paper,  a tree $T$ will always be vertex-labeled. 
\begin{definition} [$(T,q)$-linkedness]
{\rm
Let $T$ be a tree whose vertices are labeled $1,\dots, v(T)$. We will automatically view all subtrees of $T$ as labeled as well.
If  $v_1,\dots, v_p$ are the leaves of $T$ in the natural order determined by their labels, then we call $\tuple{v_1,\dots, v_p}$ the 
{\it leaf vector} of $T$.  Let $G$ be a graph.   Let $q$ be a positive integer. Given a $p$-tuple $\tuple{y_1,\dots, y_p}$ of vertices in $G$, we say that $\tuple{y_1,\dots, y_p}$ is
{\it $(T,q)$-linked} in $G$ if there exist $q$ labeled copies of $T$ in $G$ that have  leaf vector $\tuple{y_1,\dots, y_p}$ but are otherwise vertex-disjoint.
}
\end{definition}

\begin{definition}[Heavy, light, and admissible trees]
{\rm Let $L$ be a positive integer. Let $T$ be a tree on $[v(T)]$ and $G$ be a graph. We call a labeled copy $\widetilde{T}$ of $T$ in $G$ {\it $L$-light} if the number of labeled copies of $T$ in $G$ that share the same leaf vector as $\widetilde{T}$ is at most $L^{e(T)^2}$. Otherwise, we say that $F$ is {\it $L$-heavy} in $G$. We say that $\widetilde{T}$ is {\it $L$-admissible} in $G$ if every proper subtree of $\widetilde{T}$ is $L$-light in $G$.
}
\end{definition}

We also recall the  K\H{o}v\'ari-S\'os-Tur\'an theorem \cite{KST}.

\begin{lemma} [\cite{KST}] \label{lem:KST}
Let $s$ be a positive integer. If $H$ is a $K_{s,s}$-free bipartite graph with parts $X,Y$, each of size $m$.
Then $e(H)\leq (s-1)^{1/s} m^{2-1/s}+(s-1)m$.
\end{lemma}

We will also make use of the following well-known result of Erd\H{o}s.

\begin{lemma}[\cite{Erdos}]\label{lem:dense-partite}
Let $F$ be a fixed $q$-partite $q$-uniform hypergraph.
There exist positive reals $\delta_F$ and $C_F$ depending in $F$ 
such that $\ex(n,F)\leq C_Fn^{q-\delta_F}$.
\end{lemma}

The following lemma is a special case of Lemma 2.5 in \cite{CJL}.  Since our notation is slightly different, we give a proof for completeness.

\begin{lemma} \label{lem:2-paths}
Let $\varepsilon>0, K$ be real number, $s$ be a positive integer then there exists $d_0=d_0( \varepsilon, s, K)$ such that for any $L\geq (s^2+2s)^{1/4}$ and $d\geq d_0$ the following holds.  If
 $G$ is an $n$-vertex $K_{s,s}'$-free $K$-almost regular graph with minimum degree $d$
then the number of labeled $L$-heavy $2$-paths  in $G$ is at most $\varepsilon nd^2$.
\end{lemma}
\begin{proof}
For each $v$ in $V(G)$, let $F_v$ be a graph on vertex set $V(F)=N_G(v)$ such that for any pair $a,b\in V(F)$, $ab\in E(F_v)$ if and
only if $avb$ is a heavy $2$-path in $G$. Note that the number of labeled heavy $2$-paths in $G$ with central vertex $v$ is given by
$2e(F_v)$.  Observe that if $F_v$ contains a copy $D$ of $K_{s,s}$ then $G$ contains a copy of $K_{s,s}'$ by greedily replacing every
edge $ab$ of $D$ with a corresponding $2$-path in $G$ joining $a$ to $b$ that avoids the vertices already used, where here we use that each edge $ab$ corresponds to an $L$-heavy $2$-path for $L \geq (s^2 + 2s)^{1/4}$. This, however, would
contradict $G$ being $K'_{s,s}$-free. Hence $F_v$ is $K_{s,s}$-free. By Lemma \ref{lem:KST}, $e(F_v)\leq 2s|V(F)|^{2-1/s}
\leq 2s(Kd)^{2-1/s}$.

Hence, the number of labeled $L$-heavy $2$-paths in $G$ is at most $4sn(Kd)^{2-1/s}\leq \varepsilon nd^2$, as long as $d \geq (4s K^2\ve^{-1})^s$.
\end{proof}

Next, we define a few subgraphs frequently used in our proofs.

\begin{definition} \label{defn:special-graph}
{\rm 
A {\it spider} $S_p$ of $p$ legs of height two is a graph formed from a star $K_{1, p}$ by replacing every edge with a path of length $2$. Formally, $S_p$ has vertex set
$V=\{w,a_1,\dots, a_p, b_1,\dots, b_p\}$ and edge set $E=\{wa_i: i\in [p]\}\cup \{a_ib_i: i\in [p]\}$.
Let $S^-_p$ be obtained from $S_p$ by deleting $b_p$. We will view $\tuple{b_1,\dots, b_{p-1}, a_{p}}$
as the leaf vector of $S^-_p$.}
\end{definition}

\begin{definition}{\rm Given a bipartite graph $H$ with a bipartition $(V_1,V_2)$, we denote the hypergraph on $V_1$ with
edge set $\{N_H(b): b\in V_2\}$ by $\N_{V_1}(H)$ and call it the {\it neighborhood hypergraph} of $H$ induced
on $V_1$. We say that $\N_{V_1}(H)$ is $q$-partite if $V_1$ can be partitioned in $q$ subsets $A_1,\dots, A_q$
such that each member of $\N_{V_1}(H)$ contains at most vertex in each $A_i$. Note that this implies that each vertex in $V_2$ has degree at most $q$ in $H$. }
\end{definition}

\begin{definition}
Let $A, B_1,\dots, B_r$ be disjoint sets of vertices, where $A=\{a_1,\dots, a_\ell\}$ 
and $\forall i\in [r], |B_i|=t$.
Let $\Gamma_{r,t,\ell}$ denote the $(t+1)$-uniform hypergraph with edges $\{a_i\cup B_j: i\in [\ell], j\in [r]\}$.
Note in particular that $\Gamma_{r,t,\ell}$ is a $(t+1)$-partite $(t+1)$-uniform hypergraph.
\end{definition}
 Notice that the $\Gamma_{r, t, \ell}$ is a neighborhood hypergraph for the graph $F_{r, t}^{\ell}$.

\begin{lemma} \label{lem:strong-tuples-embed}
Let $q$ be a positive integer. 
Let $H$ be a bipartite graph with a bipartition $(V_1,V_2)$, where the neighborhood hypergraph
$\N_{V_1}(H)$ is $q$-partite. Let $L$ be positive integer satisfying $L\geq v(H')$.
Let $G$ be a graph. Let $\F$ be a collection of $q$-tuples of vertices in $G$ that are $(S_q,L)$-linked in $G$. 
If  $\F$, viewed as a $q$-uniform hypergraph, contains a copy of $\N_{V_1}(H)$
then $G$ contains $H'$.
\end{lemma}
\begin{proof} Suppose $V_2=\{y_1,\dots, y_m\}$.  For each $i\in [m]$, let $e_i=N_{V_1}(y_i)$.
By our assumption, there is a copy $N^*$ of $\N_{V_1}(H)$
in $G$ whose edges are $(S_q,L)$-linked tuples. For each $i\in [m]$, let $f_i$ denote the image of $e_i$ in $N^*$.
By our assumption, for each $i\in [m]$, there are $L$ internally disjoint copies of $S_q$ with leaf set $f_i$.
Since $L\geq v(H')$, one can find $m$ copies of $S_q$: $T_1,\dots, T_m$ such that
for each $i\in [m]$, the leaf set of $T_i$ is $f_i$ and that the sets of non-leaves of the $T_i$'s
are pairwise disjoint. Now $\bigcup_{i=1}^m T_i$ forms a copy of $H'$ in $G$.

\end{proof}

Next, similar to Lemma~\ref{lem:2-paths}, we show that an almost regular  $H_{r, t}^{\ell}$-free graph does not
contain too many $L$-heavy spiders $S_{t + 1}$ or too many $L$-heavy $S^-_{t+1}$. We prove the statement
about $S^-_{t+1}$ first, since it will be used to prove the statement about $S_{t+1}$.

\begin{definition} [$D$-shadows] \label{defn:D-shadow}
{\rm Let $T$ be a labeled tree and $D$ a proper subtree. 
Given a family $\F$ of labeled copies of $T$ and $F\in \F$, we
call the image of $D$ in $F$ the {\it $D$-shadow of $F$} and denote it by $F[D]$.
We let $\partial_D(\F)=\{F[D]: F\in \F\}$ and call it the {\it $D$-shadow} of $\F$.
}
\end{definition}

The next lemma about $S^-_{t+1}$ is a somewhat involved and uses a non-standard averaging 
together with linkage arguments.

\begin{lemma} \label{lem:short-spiders}
Let $r,t,\ell$ be fixed positive integers.
Let $\varepsilon>0$ be a real. There exists a positive integer $L_1 = L_1(r,t, \ell)$ such
that for all $L\geq L_1$  and all sufficiently large $n$ the following holds.
There exists a $d_0 = d_0(\ve, r, t, L)$ such that for any $n$-vertex $H_{r,t}^\ell$-free $K$-almost regular graph $G$ with minimum degree $d \geq d_0$, the number of $L$-heavy copies of $S_{t+1}^-$ in $G$ is at most $\varepsilon nd^{2t+1}$.
\end{lemma}
\begin{proof}
Let $\F$ denote the family of $L$-heavy copies of $S_{t+1}^-$ in $G$.
Suppose for contradiction that $|\F|>\varepsilon nd^{2t+1}$.
Let $M$ denote the number of different leaf vectors involved in members of $\F$.
Since members of $\F$ are $L$-heavy in $G$, we have 
$|\F|\geq M\cdot L^{[e(S_{t+1}^-))^2]}= ML^{(2t+1)^2}$.
Hence, 
\begin{equation}\label{eq:leaf-vector-count}
M\leq |\F|/L^{(2t+1)^2}.
\end{equation}

Let $\F_1$ denote the subfamily of members of $\F$ that contain a $L$-heavy $2$-path.
Let $s=\ell r (t+1)$. Clearly, $G$ is $K'_{s,s}$-free.
By Lemma \ref{lem:2-paths},
by choosing $L_1 \geq \left((\ell(1 + r) + rt)^2 + 2(\ell(1 + r) + rt)\right)^{1/4}$, we can ensure that the number of $L$-heavy $2$-paths in $G$ is at most 
$(\varepsilon/2K^{2t-1})nd^2$. Since $\Delta(G)\leq Kd$, we then have
$|\F_1|\leq (\varepsilon/4K^{2t-1})nd^2)(Kd)^{2t-1}\leq (\varepsilon/4) nd^{2t+1}\leq |\F|/4$.

Let $\F_2=\F\setminus \F_1$. Then $|\F_2|> (3/4)|\F|$. 
By definition, members of $\F_2$ do not contain any $L$-heavy $2$-path in $G$. 

Let $D$ denote the $t$-legged spider obtained from $S_{t+1}^-$ by deleting the leg of length one.
We next clean $\F_2$ by the following two rules
\begin{enumerate}
\item Whenever there is a member $F$ of $\F_2$ such that $F[D]$ is the $D$-shadow of 
fewer than $L$ remaining members of $\F_2$ then delete all such members of $\F_2$.

\item Whenever there is a member $F$ of $\F_2$ such that its leaf vector is the leaf vector of fewer than
$(1/2)L^{(2t+1)^2}$ remaining members of $\F_2$ then remove all the members of $\F_2$ with this
leaf vector.
\end{enumerate}

We continue until no more removals can be performed. Let $\F_3$ denote the remaining subfamily of $\F_2$.
Since there are at most $n(Kd)^{2t}$ different choices of copies of $D$ in $G$, the number of members removed 
by type 1 removals is at most $Ln(Kd)^{2t}$. 
The number of members removed by type 2 removals is at most $M(1/2) L^{(2t+1)^2}\leq |\F|/2$,
where we used \eqref{eq:leaf-vector-count}. Hence,
\[|\F_3|\geq |\F_2|-Ln(Kd)^{2t}-|\F|/2\geq (1/4)|\F|\geq (\varepsilon/4)nd^{2t+1},\]
for $n$  sufficiently large and $d \geq 4L \ve^{-1}$. 
By our assumption, for each $F\in \F_3$ the number of members of $\F_3$ whose leaf vector
matches the leaf vector $L(F)$ of $F$ is at least $(1/2)L^{(2t+1)^2}$. 

\begin{claim} \label{claim:LF-linked}
$L(F)$ is $(S_{t+1}^-,L)$-linked in $\F_3$. 
\end{claim}\begin{proof}
Let $\cH$ the family of copies of $S_{t+1}^-$ that have leaf vector $L(F)$.
Take a maximal collection $\M$ of vertex disjoint members of $\cH$. Let $M = \bigcup_{F \in \M} V(F) \setminus L(F)$. Since $\M$ is maximal, we have that every $F$ in $\cH$ uses some vertex in $M$. We will use this to give a lower bound on $|M|$, and use it to show $|\M|\geq L$.

For each $v \in M$, since each member of $\F_3$ uses only $L$-light $2$-paths, there are at most $L^{4t}$ copies of $S_{t + 1}^-$ in $\cH$ maping $w$ to $v$. Similarly, the number of members of $\cH$ which map $a_i$ to $v$ is at most $L^{4t}$.
Thus, $|\F_3| \leq |M| (t + 1) L^{4t}$. Since $|\M|=|M|/2(t+1)$, we have $|\M| \geq \frac{1}{2(t + 1)^2} L^{(2t + 1)^2 - 4t} \geq L$. 
\end{proof}

Next, we prove a few additional claims about $\F_3$.
Recall that $\partial_D(\F_3)=\{F[D]: F\in \F_3\}$.
Since each member of $\partial_D(\F_3)$
is contained in at most $Kd$ members of $\F_3$,  we have 

\begin{claim}
$|\partial_D(\F_3)|\geq |\F_3|/Kd\geq (\varepsilon/4K) nd^{2t}$.

\end{claim}

Recall that by the cleaning rules used to obtain $\F_3$ from $\F_2$, for each $F\in \F_3$,
$F[D]$ is the $D$-shadow of at least $L$ members of $\F_3$.

\begin{claim} For each $F\in \F_3$, the image of the tuple $\tuple{b_1,\dots, b_t, w}$ 
is $(S_{t+1},\floor{\frac{L}{t+2}})$-linked in $G$.

\end{claim}
\begin{proof} 
Let $b'_1,\dots, b'_t,w'$ denote the images of $b_1,\dots, b_t,w$ in $F$, respectively.
Note that $F[D]$ is  a $t$-legged spider with center $w'$ and leaf vector $\tuple{b'_1,\dots, b'_t}$.
Let $a^1_{t+1}$ denote the image of $a_{t+1}$ in $F$. We wish to show that
$\tuple{b'_1,\dots, b'_t,w'}$ is  $(S_{t+1},\floor{\frac{L}{t+2}})$-linked in $G$.

By Claim~\ref{claim:LF-linked}, $\tuple{b_1',\dots, b_t', a_{t+1}^1}$ is
$(S_{t+1}^-,L)$-linked. We can thus find a copy $T_1$ of $S_{t+1}^-$ whose
leaf vector is $\tuple{b_1',\dots, b_t', a_{t+1}^1}$ such that $V(T_1)$ avoids $w'$.
Now, $T'_1:=T_1\cup w'a_{t+1}^1$ is
a copy of $S_{t+1}$ with leaf vector $\tuple{b_1',\dots, b_t', w'}$,
By the earlier remark, $F[D]$ is the $D$-shadow of at least $L$ members of $\F_3$.
So, we find another member $F_2$ of $\F_3$ with $F_2[D]=F[D]$ but that $F_2$ maps $a_{t+1}$ to 
a vertex $a_{t+1}^2$ outside $V(T'_1)$.
We are able to do so since $L$ is sufficiently large in terms of $r, t, \ell$.
Now, since $\tuple{b_1', \dots, b_t', a_{t+1}^2}$ is the leaf vector of $F_2\in \F_3$,
by Claim~\ref{claim:LF-linked}, it is $(S_{t+1}^-,L)$-linked.
We can thus find a copy $T_2$ of $S_{t+1}^-$ whose
leaf vector is $\tuple{b_1', \dots, b_t', a_{t+1}^2}$ such that $V(T_2)$ share
only $b_1', \dots, b_t'$ in common with $T'_1$.
Now, $T'_2:=T_2\cup w' a_{t+1}^2$ is
a second copy of $S_{t+1}$ with leaf vector $\tuple{b_1',\dots, b_t', w'}$ that is
internally disjoint from $T'_1$.  We can continue like this to find at least $\floor{\frac{L}{t+2}}$ 
internally disjoint copies of $S_{t+1}$ with leaf vector $\tuple{b_1',\dots, b_t', w'}$.
\end{proof}

Since $|\partial_D(\F_3)|\geq (\varepsilon/4K) n d^{2t}$, by averaging there exists some vertex $v$ such that the number of
members of $\partial_D(\F_3)$ that maps $a_1$ to $v$ is at least $|\partial_D(\F_3)|/n\geq (\varepsilon/4K) d^{2t}$.
Let $\D^*$ denote the subfamily of the members of $\partial_D(\F_3)$ that map $a_1$ to $v$.
Let $B_1$ denote the set of images of $b_1$ in the members of $\D^*$.  Then $|B_1|\leq Kd$. 
For each $u\in B_1$, let $\D^*_u$ denote the subfamily of members of $\D^*$ that
map $b_1$ to $u$ and let $\cL(u)=\{T-\{v,u\}: T\in \D^*_u\}$.
Then $|\cL(u)|=|\D^*_u|$. Note that each member of $\cL(u)$ is a copy of $S_{t-1}$.
Since $\Delta(G)\leq Kd$, each vertex of $G$ lies in at most $(Kd)^{2t-2}$ members of
$\cL(u)$. Hence, for each $u$ with $|\cL(u)|\geq 2rt(Kd)^{2t-2}$, the number of tuples
$\tuple{T_1,\dots, T_r}$, where $T_1,\dots, T_r$ are vertex disjoint members of $\cL(u)$
is at least $(|\cL(u)|/2)^r$. Let $U$ denote the set of vertices $u\in B_1$ with $|\cL(u)|\geq 2rt(Kd)^{2t-2}$.
Then 
\[\sum_{u\in U} |\cL(u)|\geq (\varepsilon/4K)d^{2t} - (Kd) 2rt(Kd)^{2t-2}\geq (\varepsilon/8K)d^{2t},\]
when $d$ is sufficiently large. By convexity, the number of tuples $\tuple{u,T_1,\dots, T_r}$ where
$u\in U$ and $T_1,\dots, T_r$ are vertex disjoint members of $\cL(u)$ is at least

\[\sum_{u\in U} |\cL(u)|^r \geq |U|\left(\frac{\sum_{u\in U} |\cL(u)|}{|U|}  \right)^r\geq (Kd) 
\left[\frac{(\varepsilon/8K)d^{2t}}{Kd}\right]^r\geq  (\varepsilon/8)^r K^{1-2r}d^{(2t-1)r+1}.\]

On the other hand, the number of choices for any $T_i$
is at most $(Kd)^{2t-1}$, since there are at most $Kd$ choices to map $w$ to and at most $(Kd)^{2t-2}$
ways to grow $T_i$. So the number of choices for the tuple $\tuple{T_1,\dots, T_r}$ is at most
$(Kd)^{(2t-1)r}$. By averaging, there exists a tuple $\tuple{T_1,\dots, T_r}$ such that there are at least 
$K(\varepsilon/8K^2)^r d\geq \ell$ different vertices $u_1,\dots, u_\ell$ such that for each $i\in [\ell]$,
$T_1,\dots, T_r$ are vertex disjoint members of $\cL(u_i)$.

For each $j\in [r]$, let $w^j$ denote the image of $w$ in $T_i$ and $b^j_1,\dots, b^j_t$ denote the images
of $b_2,\dots, b_t$ in $T_j$, respectively. By Claim 3, for each $i\in [\ell], j\in [t],\tuple{u_i, b^j_2,\dots, b^j_t, w^j}$
is $(S_{t+1}, \floor{\frac{L}{t+2}})$-linked in $G$. Note that the $(t+1)$-uniform hypergraph corresponding to these
$rt\ell$ tuples is a copy of $\Gamma_{r,t,\ell}$. For $L$ sufficiently large, we have $L\geq v(H^\ell_{r,t})$ and 
by Lemma \ref{lem:strong-tuples-embed}, $G$ contains a copy of $H^\ell_{r,t}$, contradicting $G$ being $H^\ell_{r,t}$-free.
Hence, the number of $L$-heavy copies of $S_{t+1}^-$ is at most $\varepsilon nd^{2t+1}$.
\end{proof}

\begin{lemma} \label{lem:spider}
Let $\varepsilon>0$ be a real and $r, t, \ell \geq 1$ integers. There exists a positive integer $L_2$ depending only on $r, t, \ell $ such
that for all $L\geq L_2$ and all sufficiently large $n$ the following holds.
Let $d \geq d_0 = d_0(\ve, r, t, \ell, L)$, and $G$ be an $n$-vertex $H_{r,t}^\ell$-free $K$-almost regular graph with minimum degree $d$.
Then the number of $L$-heavy copies of $S_{t +1}$ is at most $\varepsilon nd^{2(t + 1)}$. 
\end{lemma}
\begin{proof}
Let $\F$ denote the family of $L$-heavy copies of $S_{t+1}$ in $G$. Suppose for contradiction that 
$|\F|>\varepsilon nd^{2t+2}$. 
Let $M$ denote the number of different leaf vectors involved in members of $\F$.
Since members of $\F$ are $L$-heavy in $G$, we have 
$|\F|\geq M\cdot L^{[e(S_{t+1}))^2]}=M \cdot L^{(2t+2)^2}$.
Hence, 
\begin{equation}\label{eq:leaf-vector-count2}
M\leq |\F|/L^{(2t+2)^2}.
\end{equation}

 Let $\F_1$ denote the subfamily of members of $\F$ that contain a $L$-heavy $2$-path
and let $\F_2$ denote the subfamily of members  of $\F$ that contains a $L$-heavy copy of $S_{t+1}^-$.
By Lemma \ref{lem:2-paths} and Lemma \ref{lem:short-spiders}, by choosing $L$ to be sufficiently large, we can ensure
that the number of heavy $2$-paths in $L$ is at most $(\varepsilon/8K^{2t})nd^2$ and that the number of
heavy $S_{t+1}^-$ is at most $(\varepsilon/8K) nd^{2t+1}$. Since $\Delta(G)\leq Kd$, we have
$|\F_1|\leq (\varepsilon/4K^{2t}) n d^2 (Kd)^{2t}\leq (\varepsilon/8) nd^{2t+2}$ and $|\F_2|\leq (\varepsilon/8K) nd^{2t+1}(Kd)
\leq (\varepsilon/8) nd^{2t+2}$.

Let $\F_3=\F\setminus (\F_1\cup \F_2)$. Then $|\F_3|\geq (3/4)|\F|$.

\medskip

 We now clean $\F_3$ as follows. Whenever there is a member $F$ such that the number of remaining members of $\F_3$ that share the same leaf vector as $F$ is less than $(1/2)L^{(2t+2)^2}$,
we remove all the members which contain this leaf vector. Let $\F_4$ denote the remaining subfamily of $\F_3$ after no more removals can be performed. The number of members removed is at most $|M|(1/2)L^{(2t+2)^2}<|\F|/2$, where we used \eqref{eq:leaf-vector-count2}.
Hence, $|\F_4|\geq |\F|/4\geq (\varepsilon/4) nd^{2t+2}$.  By definition, for each $F\in \F_4$, there are at least $\frac{1}{2}L^{(2t+2)^2}$
members of $\F_4$ having the same leaf vector as $F$.

\medskip

\begin{claim}The leaf vector of each member of $\F_4$ is  $(S_{t+1},L)$-linked in $G$. 
\end{claim}
\medskip

\begin{proof} Let $F\in \F_4$. Let $Z=\tuple{z_1,\dots, z_{t+1}}$
be the leaf vector of $F$. Let $\F_Z$ be the subfamily of members of $\F_4$
that have the leaf vector $Z$.
Let $F_1,\dots, F_m$ be a maximum collection of
internally disjoint members of $\F_Z$. We may assume
$m<L$; otherwise we are done. Let $U=\bigcup_{i=1}^m V(F_i)\setminus Z$.
Then $|U|\leq (t+2)m<(t+2)L$. By maximality, each member of $\F_Z$ 
must map one of $\{w,a_1,\dots, a_{t+1}\}$ to a vertex of $U$.  Recall that members of $\F_Z$ contain
no $L$-heavy $2$-paths in $G$ and no $L$-heavy $S_{t+1}^-$ in $G$.
Consider any $u\in U$, the number
of members of $\F_Z$ that map the central vertex $w$ of $S_{t+1}$
to $u$ is at most $(L^4)^{t+1}=L^{4t+4}$.  For any $i\in [t+1]$,
the number of members of $\F_Z$ that map $a_i$ to $u$ is at most
$L^{(2t+1)^2}$, because members of $\F_Z$ contain no $L$-heavy $S_{t + 1}^-$. So $|\F_Z|\leq |U|(L^{4t+4}+(t+1)L^{(2t+1)^2})
<(t+2)^2L\cdot L^{(2t+1)^2}<\frac{1}{2} L^{(2t+2)^2}$, provided that
$L\geq L_2$ and $L_2$ is sufficiently large in terms of $t$. But this contradicts our
earlier claim about members of $\F_4$.
\end{proof}

Since $|\F_4|\geq (\varepsilon/4)nd^{2t+2}$ by averaging, there exists a vertex $v$ such that
the number of members of $\F_4$ that map the central vertex $w$ of $S_{t+1}$ to $v$ is at least $(\varepsilon/4) d^{2t+2}$.
Let $\D$ denote the subfamily of these members having $v$ as the central vertex.
Let $N^2_G(v)$ denote the set of vertices at distance at most two from $v$ in $G$. Then $|N^2_G(v)|\leq (Kd)^2$. 
Let $\cL$ be the $(t+1)$-uniform hypergraph on $N^2_G(v)$ whose hyperedges correspond to the leaf vectors of $\D$. Note that each hyperedge corresponds to no more than $(t + 1)! L^{4t + 4}$ copies of $F$ in $\F_4$ as no member of $\F_4$ contains an $L$-heavy $2$-path. 
Then $v(\cL)\leq (Kd)^2$ and $e(\cL)\geq \frac{\varepsilon}{4(t+1)! L^{4t + 4}} d^{2t+2}\geq \frac{\varepsilon}{4(t+1)!K^{2t+2}L^{4t + 4}} (v(\cL))^{t+1}$.
Recall that $\Gamma_{r,t,\ell}$ is $(t+1)$-partite and by Lemma~\ref{lem:dense-partite}, $\ex(m,\Gamma_{r,t,\ell})=o(m^{t+1})$.
Hence, when $d$ is sufficiently large (thus $v(\cL)$ is sufficiently large), we can ensure that $\cL$ contains a copy of $\Gamma_{r,t,\ell}$.
For $L$ sufficiently large, we have $L\geq 2r\ell(t+1)$. Since $\Gamma_{r,t,\ell}$ is a neighborhood hypergraph
of $F^\ell_{r,t}$, by Lemma \ref{lem:strong-tuples-embed}, $G$ contains a copy of $H^\ell_{r,t}$, a contradiction.
Hence, the number of $L$-heavy copies of $S_{t+1}$ is at most $\varepsilon nd^{2t+2}$.

\end{proof}

\section{Strong trees and anchored subfamilies} \label{sec:strong-anchor}

In this section, we introduce two important notions and a few lemmas related to these notions.

\begin{definition}[strong trees]\label{def:strong}
{\rm 
Let $L$ be a positive integer.
Let $T$ be a tree. We call a labeled copy $F$ of $T$ in a graph $G$ {\it $L$-strong} 
if its leaf vector is $(T,L)$-linked in $G$.
Let $G$ be a graph and $\cH$ be a family of labeled copies of $T$ with 
the same leaf vector $Z$. We say that $\cH$ is {\it $L$-strong} if there exists $L$
members of $\cH$ that are pairwise vertex disjoint outside the common leaf vector $Z$.
}
\end{definition}

For the next lemma, recall the notations related to $T'_{r,t}$, given in Definition~\ref{def:Trt}.
\begin{lemma} \label{lem:strong-subtree-lemma}
Let $r,t,\ell$ be positive integers, and $K$ a positive real.
Let $D$ be a subtree of $T'_{r,t}$ with $rt$ leaves obtained from $T'_{r,t}$ by
deleting at most one $z_{i,j}$ for each $i\in [r]$.   
There exist positive constants $C = C(r, t, \ell, K, D)$ and $L_3$ such
that for all $L\geq L_3$ and all sufficiently large $n$ the following holds.
Let $d \geq d_0(r, t, \ell, K, D, L)$, and $G$ be an $n$-vertex $H_{r,t}^\ell$-free $K$-almost regular graph with minimum degree $d$.
Then the number of $L$-strong copies of $D$ in $G$ is at most $C nd^{e(D)-1}$.
\end{lemma}
\begin{proof}
Let $C$ be sufficiently large in terms of $r, t, \ell, K,$ and $e(D)$ and let $L \geq L_3$ with $L_3$ sufficiently large in terms of $r, t, \ell$.  
Let $\F$ be the family of all $L$-strong copies of $D$ in $G$. Suppose for contradiction that
$|\F|>C nd^{e(D)-1}$. Suppose $e(D)=rt-p$, where $1\leq p\leq r$. Without loss of generality, suppose $D=T'_{r,t}-\{z_{1,1}, z_{2,1},\dots, z_{p,1}\}$. 

First, we clean $\F$ to obtain a subfamily as follows. Whenever there is a subtree $B$ of $D$
of the form $D-v_{i,1}$ for some $i\in [p]$ and a member $F\in \F$ such that $F[B]$ is the $B$-shadow
of at most $\frac{C}{2K^{e(D)}p}$, we remove all of these members of $\F$. We stop when no more removals can be performed. Clearly, there are at most $n(Kd)^{e(D)-1}p$ choices of $B$. Hence, the number of
members of $\F$ that are removed is at most $\frac{C}{2K^{e(D)}}n(kd)^{e(D)-1}p\leq \frac{C}{2} n d^{e(D)-1}$, when $n$ is sufficiently large. Let $\F_1$ denote the remaining subfamily of $\F$.
Then $|\F_1|\geq \frac{1}{2} nd^{e(D)-1}$. Let $F$ be any member of $\F_1$.
For each $i\in [r]$, let $y'_i$ denote the image of $y_i$ in $F$. For each $i\in [r], j\in [t]$,
let $z'_{i,j}$ denote the image of $z_{i,j}$ in $F$ and let $v'_{i,j}$ denote the image of $v_{i,j}$ in $F$. 

\begin{claim}
Let $Z^*$ denote the $rt$-tuple obtained from $\tuple{z'_{1,1},\dots, z'_{1,t}, \dots, z'_{r,1},
\dots, z'_{r,t}}$ by replacing $z'_{i,1}$ with $y'_i$ for each $i\in [p]$.
Let $S$ be any set of at most $(2rt+2r+1)\ell$ vertices outside $Z^*$.
Then there exists
a labeled copy of $T'_{r,t}$ in $G$ with leaf vector $Z^*$ that avoids $S$.
\end{claim} 
\begin{proof} Note that by our cleaning rule, there are at least $\frac{C}{2K^{e(D)p}} \geq (2rt + 2r + 1)\ell$ members of $\F_1$ that have the same image of $D-v_{1,1}$ as $F$. Among these, we can choose one, $F_1$, that maps $v_{1,1}$ to a vertex $v''_{1,1}$ outside $S\cup\{v'_{1,1}\}$. By a similar reasoning, there is a member $F_2$ of $\F_1$ that maps $v_{2,1}$ to a vertex $v''_{2,1}$ outside $S$ but are otherwise identical to $F_1$. We can continue like this
to find a member $F_p$ such that for each $i\in [p]$, $F_p$ maps $v_{i,1}$
to a  vertex $v''_{i,1}$ outside $S$ but $F_p$ is otherwise identical to $F_{p-1}$.
Now since $F_p\in \F_1\subseteq \F$. By our assumption of $\F$ there are at least $L$ internally
disjoint members of $\F$ that have the same leaf vector as $F_p$. We can find a member $F'$
that avoids $S\cup \{y'_1,\dots y'_p\}$. Now, $F'\cup \{y'_1v''_{1,1}, y'_2v''_{2,1},\dots, y'_pv''_{p,1}\}$ is a copy of $T'_{r,t}$ with leaf vector $Z^*$ that avoids $S$, by our assumption that $L \geq (rt + 2r + 1)\ell$.
\end{proof}

Now, we can apply this property $\ell$ times to find $\ell$ internally disjoint copies of $T'_{r,t}$
with leaf vector $Z^*$. Namely, first we let $S=\emptyset$ and apply the claim to find a copy
$T_1$ of $T'_{r,t}$ with leaf vector $Z^*$. Next, we let $S=V(T_1)\setminus Z^*$ and apply 
the claim to find a copy $T_2$ of $T'_{r,t}$ with leaf vector $Z^*$ that avoids $S$. Next, we
let $S=(V(T_1)\cup V(T_2))\setminus Z^*$ and apply the claim again to find $T_3$. We can continue
like this to find $\ell$ internally disjoint copies of $T'_{r,t}$ with leaf vector $Z^*$. Now, now
the union of them forms a copy of $H_{r,t}^\ell$ in $G$, contradicting $G$ being $H_{r,t}^\ell$-free.
\end{proof}


\begin{lemma} \label{lem:strong-spiders}
Let $p,s,L\geq2$ be integers. Let $\varepsilon>0$ be any real.
Let $H$ be a bipartite graph with a bipartition $(V_1,V_2)$ such that
the neighborhood hypergraph $\N_{V_1}(H)$ in $V_1$ is $q$-partite.
For all $L \geq v(H')$, there exists a constant $C_{1}$ depending on $H, L$ such that the following holds.
Let $G$ be an $H'$-free graph. Let $w\in V(G)$. Let $\cP_w$ be a family of
$L$-light $2$-paths starting at $w$. Suppose $|\cP_w|\geq C_{1}$ and each vertex
in $G$ is the internal vertex of at most $\frac{|\cP_w|}{2q^2}$ members of $\cP_w$.
Let $\F_w$ be the family of all labeled copies of $S_q$ in $G$ whose legs are
members of $\cP_w$. Then
\begin{enumerate}
\item $|\F_w|\geq \frac{1}{2}|\cP_w|^q$.
\item The number of members of $\F_w$ that are $L$-strong in $G$ is at most $\varepsilon |\F_w|$.
\end{enumerate}
\end{lemma}
\begin{proof}
First, we choose $C_{1}$ to be large enough so that $\ex(m, \N_{V_1}(H))\leq \frac{\varepsilon}{2q!L^{4q}} m^q$
for all $m\geq C_{1}/L^4$. Such $C_{1}$ exists by Lemma~\ref{lem:dense-partite}.
The number of ways to pick members $P_1,\dots, P_q\in \cP_w$ is 
$|\cP_w|^q$. The number of such  $q$-tuples in which for some $i\neq j$, $P_i, P_j$ have the same second vertex
is at most $|\cP_w|^{q-1}\cdot {q \choose 2} \frac{|\cP_w|}{2q^2} <\frac{1}{4}|\cP|^q$, by our assumption.
Since members of $\cP_w$ are $L$-light $2$-paths,
the number $q$-tuples such that for some $i\neq j, P_i, P_j$ have the same third vertex
is at most $|\cP_w|^{q-1}\cdot {q \choose 2} L^4<\frac{1}{4}|\cP_w|^q$. Therefore, the number of $q$-tuples $P_1,\dots, P_q$
that form a copy of $S_q$ is at least $\frac{1}{2}|\cP_w|^q$. By our assumption, $\F_w$ is
the family all labeled copies of $S_q$ whose legs are members of $\cP_w$. By our prior discussion,
$|\F_w|\geq \frac{1}{2}|\cP_w|^q$.

Let $V$ be the set of endpoints of members of $\cP_w$ opposing $w$. Note that since members of $\cP_w$ are $L$-light, $|V|\geq |\cP_w|/L^4\geq C_{1}/L^4$. Also, trivially, $|V|\leq |\cP_w|$. Let $\D$ be the $q$-uniform hypergraph on $V$ whose edges are the unordered $q$-sets corresponding to the leaf vectors of members of $\F_w$ that are $L$-strong in $G$. Suppose for contradiction that the number of members 
of $\F_w$ that are $L$-strong in $G$ is more than $\varepsilon |\F_w|$. 
Note that each $s$-set can correspond to the leaf vector of at most $q!L^{4q}$ members of $\F_w$, since legs of any member
of $\F_w$ are $L$-light and there are $q!$ permutations of the $s$-set.
So, $e(\D)\geq 
\frac{\varepsilon}{q!L^{4q}} |\cP|^q\geq \frac{\varepsilon}{q!L^{4q}}|V|^q>\ex(|V|, \N_{V_1}(H))$. Note that as $|V| \geq C_{1} / L^4$, the bound holds. 
Hence $\D$ contains a copy
$\N^*$ of $\N_{V_1}(H)$. Since edges of $\N^*$ correspond to leaf vectors of $L$-strong copies of $S_q$ in $G$, and $L \geq v(H')$,
we can greedily embed $H'$ into $G$ with $N_{V_1}(H)$ mapped to $\N^*$, contradicting $G$ being
$H'$-free. 
\end{proof}

\begin{definition} [anchors and anchored subfamilies]
{\rm Let $T$ be a tree with $q$ leaves. Let $u$ be a non-leaf of $T$. Let $G$ be a graph.
Let  $Z\in [V(G)]_q$. Let $\mathcal{F}_Z$ be a family of labeled copies of $T$ with leaf vector $Z$. 
If there exists a vertex $v$ in $G$ such that
the subfamily $\A_Z$ of members of $\F_Z$ that map $u$ to $v$ has size at least $\alpha |\mathcal{F}_Z|$, we call $v$ a {\it $(\alpha, u)$-anchor} for $\mathcal{F}_Z$ and we call $\A_Z$ the  {\it $(\alpha,u,v)$-anchored subfamily} of $\F_Z$.}
\end{definition}
\begin{lemma}\label{lem:anchored}
Let $L$ be a positive integer. Let $T$ be a tree with $q$ leaves. Let $\alpha$ be a positive
real with $\alpha\leq \frac{1}{(v(T)-q)^2 L}$.
Let $G$ be a graph and  $Z$ be a vector of $q$ vertices in $G$.
Let $\F_{Z}$ be a family of labeled copies of $T$ with leaf vector $Z$, such that the family $\F_Z$ is not $L$-strong.  
Then there exists a non-leaf vertex $u$ of $T$ and  a vertex $v$ in $G$
such that $v$ is an $(\alpha, u)$-anchor for $\F_Z$.
\end{lemma}
\begin{proof}
 Consider a maximal collection $\M_Z$ of members of $\F_Z$ that are pairwise
vertex-disjoint outside of $Z$.  Let $S=\left(\bigcup_{F\in \M_Z} V(F)\right)\setminus Z$. Since  $\F_Z$ is not $L$-strong, $|M_Z|<L$. Therefore, $|S|\leq (v(T) - q) L$.
By the maximality of $\M_Z$, every member of $\F_Z$ must contain a vertex in $S$. 
By the pigeonhole principle, there is a vertex $v$ in $S$ that lies in at least $\frac{|\F_Z|}{(v(T) - q) L}$  of the members of $\F_Z$. Applying the pigeonhole
principle  again,  there exists a nonleaf $u$ of $T$ such that at least $\frac{|\F_Z|}{(v(T) - q)^2L}$  members of $\F_Z$ map $u$ to $v$. 
\end{proof}


\section{Embedding lemmas for subdivisions} \label{sec:subdiv-embed}

In this section, we develop some useful lemmas for embedding subdivisions into a host graph.
The first lemma is a standard. The next two are new and all three provide
key ingredients to the proof of Theorem \ref{thm:main}.
We start with the following lemma whose proof follows from Lemma 3.4 of \cite{CJL} or  Lemma 2.8 of \cite{CJ},
and is quite standard. Nevertheless, since our statement is slightly different, we include a proof for completeness.

\begin{lemma} \label{lem:subdiv-lem1}
Let $H$ be a bipartite graph with a bipartition $(A,B)$ such that each vertex in $B$
has degree at most two. There is a constant $C_{2}=C_{2}(H)$ depending on $H$ such that
the following holds. Let $G$ be a bipartite graph with a bipartition $(X,Y)$.
Let $\delta_Y$ denote the minimum degree of a vertex in $Y$.
Suppose $\delta_Y\geq C_{2}$ and that the number of labeled $2$-paths in $G$ centered in $Y$ is 
at least $C_{2}|X|^2$. Then $G$ contains $H$.  In particular, if $\delta_Y\geq C_{2}$
and $e(G)\cdot \delta_Y\geq 2 C_{2}|X|^2$ then $G$ contains $H$.
\end{lemma}
\begin{proof}
Let $a=|A|, b=|B|$. Let $C_{2}=\binom{a}{2} b$.
We call a $2$-path $xyz$ with $x,z\in X$ {\it heavy} if $|N_G(x)\cap N_G(z)|\geq b$,
otherwise call it {\it light}. For each vertex $y\in Y$, let $\lambda_H(y), \lambda_L(y)$
denote the number of heavy and light $2$-paths centered at $y$, respectively.
By our assumption, $\sum_{y\in Y} [\lambda_H(y)+\lambda_L(y)] \geq C_{2}|X|^2$.
On the other hand, by definition, $\sum_{y\in Y} \lambda_L(y)\leq b\binom{|X|}{2}<b|X|^2$.
Hence, there exists some $y\in Y$ such that $\frac{\lambda_H(y)}{\lambda_H(y)+\lambda_L(y)}
\geq 1-\frac{b}{C_{2}}$.

Define an auxiliary graph $D$ on $N(y)$ such that $xz\in E(D)$ if and only if $xyz$ is a heavy $2$-path.
Then $e(G)\geq (1-\frac{b}{C_{2}})\binom{|N(y)|}{2}\geq (1-\frac{1}{\binom{a}{2}} )\binom{|N(y)|}{2}$.
By Tur\'an's theorem, $D$ contains an $a$-set $U$ that induces a complete graph.
Since pairs in $U$ correspond to endpoints of heavy $2$-paths in $G$, we can find a copy of $H$ in $G$
by mapping $A$ to $U$ and grow $H$ greedily.

Now, suppose $e(G)\cdot \delta_Y\geq 2C_{2}|X|^2$. Note that the number of labeled
$2$-paths in $G$ centered in $Y$ is at least $e(G)\cdot(\delta_Y-1)\geq C_{2}|X|^2$ since $C_{2}$ is sufficiently large. By the first statement, $G$ contains a copy of $H$.
\end{proof}

Now, we are ready to prove the main lemmas (Lemma~\ref{lem:spiders-to-graph} and Lemma~\ref{lem:subdiv-lem3}) of the section. Their proofs use some novel ideas that roughly speaking combine the dependent random choice approach with
the anchored subfamily  notion. They will be applied to embed subdivisions of one-side-$q$-bounded bipartite graph
similar to how dependent random choice is applied for embedding one-side-$q$-bounded bipartite graph.

\begin{definition} {\rm 
We say a $q$-legged height two spider $S_q$ is $L$-$\emph{leg-light}$ in $G$ if all of its legs are $L$-light $2$-paths. 
}
\end{definition}


\begin{lemma}\label{lem:spiders-to-graph} Let $H$ be a bipartite graph with a bipartition $(V_1,V_2)$ such that the neighborhood hypergraph $\N_{V_1}(H)$ induced on $V_1$ is $q$-partite. For all $L \geq v(H')$, there exists a $C_{3}$ depending on $L, H$ such that the following holds. 
Let $G$ be a bipartite graph. 
Let $\Pc$ be a family of $L$-light $2$-paths in $G$.
Let $W$ be the set of first vertices over members of $\Pc$, $A$ the set of second vertices, and $B$ the set of third vertices.  Suppose for some $M > 0$ and $0<\lambda\leq \frac{1}{2q^2}$, the following conditions hold. 
\begin{enumerate}
\item Each  $w\in W$ is the first vertex for at least $M$ members of $\Pc$. 
\item For each  $w\in W$ and each $b\in B$, there are at most  $\lambda M$ members of $\cP$
that have $w$ as the first vertex and $b$ as the second vertex.
\item $|\Pc| M^{q - 1} \geq C_{3} |B|^q$. 
\item $\lambda^{-1} M \geq C_{3} |B|$.
\end{enumerate}
Then $G$ contains a copy of $H'$. 
\end{lemma}

\begin{proof}

Let $C_{2}$ be the constant returned by Lemma~\ref{lem:subdiv-lem1} applied to graph $H'$.  Fix $\ve = 1/2$. Let $C_{1}$ be as in Lemma~\ref{lem:strong-spiders} applied with $H', \ve$, and $L$. For every $w \in W$, let $\cP_w$ denote the subfamily of members in $\cP$ having $w$ as the first vertex. Note that by assumption, $|\cP_w| \geq M$. Let $C_{3}$ be sufficiently large in terms of $C_{1}$, $C_{2}$, $L$, and $v(H')$. 

Note that $\lambda > 1 / M$, and so by condition 4, $M \geq \sqrt{C_{3}}$. Therefore, assuming $C_{3} \geq C_{1}^2$, we have that $|\cP_w| \geq C_{1}$. Let $\F_w$ be the family of all labeled copies of $S_q$ in $G$ whose legs are members of $\Pc_w$. Then by Lemma~\ref{lem:strong-spiders}, which we can apply since $L \geq v(H')$ and $\lambda \leq \frac{1}{2q^2}$, we have that $|\F_w| \geq \frac{1}{2} |\Pc_w|^q$, and the number of members of $\F_w$ that are $L$-strong in $G$ is at most $\ve |\F_w|$. 
Let $\F=\bigcup_{w\in W} \F_w$.
Let $\F'$ be the subfamily of all non-$L$-strong members of $\F$. 
We thus have that 

$$|\F'| \geq \frac{1}{4} \sum_{w \in W} |\Pc_w|^q \geq \frac{M^{q - 1}}{4} |\Pc| 
\geq \frac{C_{3}}{4} |B|^q,$$
where in the last step we used condition 3.
For the rest of the proof, we will use the following notation. Given a vector $U$ of at most $q$ vertices in $V$
and any subfamily $\cH$ of $\F'$, we let $\cH_U$ denote the subfamily of members of $\cH$ that map the vector
of the first $|U|$ leaves of $S_q$ to $U$. If $|U|=q$ then $\cH_U$ is precisely the subfamily of members of $\cH$
that have leaf vector $U$.
For each $Z\in [V]_q$, we say that $Z$ is good if $|\F'_Z| \geq C_{3} / 8$. Note that $\sum_{Z \text{ good }} |\F_Z'| \geq \frac{1}{2} |\F'|$.

\begin{claim}
Let $\alpha = \frac{1}{(q + 1)^2L}$.Then for each $Z\in [V]_q$ which is good, there exists an $i \in [q]$, a vertex $v \in V(G)$ such that $v$ is an $(\alpha, a_i)$-anchor for $\F_Z'$.
\end{claim}
\begin{proof}
By Lemma~\ref{lem:anchored}, there exists a nonleaf $u$ of $S_q$, a vertex $v$ in $G$, such that $v$ is $(\alpha,u)$-anchor of $\F_Z$.
Since the legs of all members of $\F_Z$ are $L$-light $2$-paths,  there can be
at most $L^{4q}<  \frac{|\F_Z|}{(q+1)^2L} \leq |\A_Z|$ many members of $\F_Z$ that map the 
central vertex $w$ of $S_q$ to $v$. Hence, we must have $u=a_i$ for some $i\in [q]$.
\end{proof}
For each good $Z$, fix an $i$ such that there is an $(\alpha, a_i)$-anchor $v_Z$. Let $c(Z) = i$
and $\A_Z$ the subfamily of members of $\F'_Z$ that map $a_i$ to $v_Z$.
By the pigeonhole principle, there is some $i\in [q]$ such that  $\sum_{Z \text{ good }: c(Z) = i} |\F_Z'| \geq \frac{1}{2q} |\F'|$,  Fix such an $i$. Without loss of generality, assume that $i = q$. 

Let \[\F''=\bigcup \left \{\A_Z: Z\in [V]_q \mbox{ is good and } c(Z) = q\right \}.\] 
Observe that $|\F''| \geq \frac{\alpha}{2q} |\F'|$. 

We clean up $\F''$ further via  the following deletion. Let $S_{q - 1}^*$ denote the subspider of $S_q$ consisting of the first $q - 1$ legs. Whenever there is a member $F \in \F''$ such that $F[S^*_{q-1}]$ is the $S_{q - 1}^*$-shadow of fewer than $\frac{\alpha}{16q} M$ members of $\F''$, remove all such members of $\F''$.
Let $\F'''$ denote the remaining subfamily of $\F''$.  Note that by definition, $|\partial_{S^*_{q - 1}}(\F'')| \leq \sum_{w \in W} |\cP_w|^{q - 1}$. Since $|\F''| \geq \frac{\alpha}{8q} \sum_{w \in W} |\Pc_w|^q
\geq \frac{\alpha}{8q}M\sum_{w\in W}|\Pc|^{q-1}\geq \frac{\alpha M}{8q}|\partial_{S^*_{q-1}}(\F'')|$,
we have that $|\partial_{S^*_{q - 1}}(\F'')| \leq \frac{8q}{\alpha M} |\F''|$. 
Therefore, the number of members we removed in the above process is at most $|\partial_{S^*_{q-1}}(\F)| \cdot \frac{\alpha}{16q} M\leq
\frac{1}{2}|\F''|$. Hence, $|\F'''| \geq \frac{1}{2} |\F''|$. 

By the pigeonhole principle, there exists a $U \in [B]_{q - 1}$ so that $|\F_U'''| \geq |\F'''||B|^{-(q - 1)}$. 
For a vertex $v\notin U$, we call $v$ {\it good} if 
$|\F'''_{U \vee \tuple{v}}| \geq \frac{1}{2} |\F_U'''| |B|^{-1}\geq \frac{1}{2}|\mathcal{F}'''||B|^{-q}$. 
Let $\F^*_U$ be the union of $\F'''_{U \vee \tuple{v}}$ over all good $v$.
Note that $|\F^*_U| \geq \frac{1}{2}|\F_U'''|$.

Let $W^*$ be the set of images of the center $w$ over members of $\F^*_U$, and $A^*$ the set of images of $a_q$ over members of $\F^*_U$. Let $R$ be the bipartite graph with parts $W^*, A^*$ by placing an edge between $v \in W^*$ and $a \in A^*$ if there is some member $F \in \F^*_U$ that maps $w$ to $v$ and $a_q$ to $a$. 

We make two claims about $R$ which we will use in conjunction with Lemma~\ref{lem:subdiv-lem1} to find a copy of $H'$.

\begin{claim}\label{claim:eR}
$$e(R) \geq \frac{\alpha}{32q L^{4(q - 1)}}\lambda^{-1}|W^*|.$$
\end{claim}
\begin{proof}

Consider any $v\in W^*$. Then some member $F$ of $\F^*_U\subseteq \F'''_U$ maps $w$ to $v$.  Note that
$F[S^*_{q-1}]$ has center $v$ and leaf vector $U$.
By the cleaning rule we used to obtain $\F'''$, there are at least $\frac{\alpha}{16q} M$ members of $\F'''$ that
share the same $S^*_{q-1}$-shadow as $F$. Also, these members are in $\F'''_U$. Hence $|\F'''_U|\geq \frac{\alpha}{16q}|W^*|M$ and thus $|\F^*_U|\geq \frac{\alpha}{32q}|W^*|M$.

On the other hand, we can upper bound $|\F_U^*|$ as follows. Consider any $va\in E(R)$ where $v\in W^*, a\in A^*$.
Consider the members that map $w$ to $v$ and $a_q$ to $a$. Since members of $\F^*_U$ are $L$-leg-light, 
there are at most $L^{4(q-1)}$ different choices for the subtree $D$ with leaf set $U\cup\{v\}$. 
For each $D$, by condition 2 of the lemma, there are at most $\lambda M$ members of $\F$ that contain $D\cup va$.
Hence, $|\F^*_U| \leq e(R) L^{4(q - 1)} \lambda M$. Combining these two bounds gives the desired inequality. 


\end{proof}

\begin{claim}\label{claim:deltaA}
$$\delta_A \geq \max\{ C_{2}, \frac{\alpha}{32q L^{4(q - 1)}|B|}|W^*|M\}.$$
\end{claim}
\begin{proof}
Fix some $a \in A^*$, then by definition there is some good $v$ such that $a$ is the common image of $a_q$
of all members of $\F_{U \vee \tuple{v}}^*$.
Since the members of $\F$ are $L$-leg-light, $|\F_{U \vee \tuple{v}}^* |\leq L^{4(q - 1)} d_R(a)$. Thus, $d_R(a)\geq \frac{|\F^*_{U\vee\{v\}}|}{L^{4(q-1)}}$. 
Since $v$ is good and $|\F'''| \geq \frac{1}{2}|\F''|\geq \frac{\alpha}{4q}|\F'|\geq \frac{\alpha}{16q} C_{3} |B|^{q}$, we have that  $|\F^*_{U \vee \tuple{v}}| \geq \frac{1}{2}|\F'''| |B|^{-q}\geq \frac{\alpha C_{3}}{32q} $,  
and thus $d_R(a)\geq  \frac{\alpha C_{3}}{32q L^{4(q-1)}}$.

Also, in the proof of Claim~\ref{claim:eR}, we showed $|\F'''_U|\geq \frac{\alpha}{16q} |W^*|M$.
Since $v$ is good, $|\F'''_{U\vee \{v\}}|\geq \frac{1}{2}|\F'''_U||B|^{-1}\geq \frac{\alpha |W^*| M}{32 q |B|}$. 
Thus, $d_R(a)\geq \frac{\alpha |W^*| M}{32 q  L^{4(q-1)} |B|}$.
Thus, picking $C_{3}$ sufficiently large in terms of $C_{2}$, we have that the conclusion of the claim holds. 
\end{proof}

Combining these two claims and our assumption $\lambda^{-1} M \geq C_{3} |B|$, we note that 
$$\delta_A \cdot e(R) \geq \left( \frac{\alpha}{32 q L^{4(q - 1)}} \right)^2 |W^*|^2 \lambda^{-1} M |B|^{-1} \geq C_{3} \left( \frac{\alpha}{32 q L^{4(q - 1)}} \right)^2 |W^*|^2. $$
Since $C_{3}$ is sufficiently large in terms of $C_{2}$ and $\delta_A \geq C_{2}$, we may apply Lemma~\ref{lem:subdiv-lem1} to $R$ and find a copy of $H'$, as desired. 

\end{proof}

We now apply Lemma~\ref{lem:spiders-to-graph} to obtain the following embedding lemma
in the asymmetric bipartite setting. We would like to mention that for the proof of 
our main theorem, it is crucial that we develop such an embedding lemma in the asymmetric
setting. The statement and its proof for the balanced version (see Corollary~\ref{cor:subdiv}) would be much simpler.

\begin{lemma} \label{lem:subdiv-lem3}
Let $K_0>0$ be a real. Let $q\geq 2$ be an integer. Let $H$ be a bipartite graph
with a bipartition $(V_1,V_2)$ such that the neighborhood hypergraph $\N_{V_1}(H)$ induced on $V_1$
is $q$-partite. 
There is a constant $C_{4}=C_{4}(H,K_0)$ such that the following holds. Let $G$ be an $n$-vertex bipartite graph with a bipartition $(X,Y)$. Let $\delta_X,\delta_Y$ denote 
the minimum degree of a vertex in $X$ and in $Y$, respectively. 
Suppose 
\begin{enumerate}
\item $\delta_X\geq C_{4}\log n$ and $\delta_Y\geq C_{4}$,
\item $\forall x\in X, \delta_X\leq d(x)\leq K_0 \delta_X$,
\item $\delta^q_X \delta^q_Y\geq C_{4} |X|^{q-1} (\log n)^q$, 
\item $\delta_X^2\delta_Y\geq C_{4}|X| (\log n)^{2}$.
\end{enumerate}
Then $G$ contains $H'$.
\end{lemma}

\begin{proof}
We may assume that $G$ is $H'$-free; otherwise we are done. 

Let $L = v(H')$. Let $C_{3}$ be as in Lemma~\ref{lem:spiders-to-graph} for $H', L$. Let $C_{4}$ be sufficiently large in terms of $C_{3}, q,  K_0$. 


\begin{claim}
There exist a subgraph $\widehat{G}$ with bipartition $(A, B)$ and  a family $\mathcal{F}$ of labeled copies of $S_q$ in $\widehat{G}$ centered in $A$ such there exist some $d_A, d_B$ such that $d_A=K_0\delta_X$ and $d_B\geq \delta_Y$, and $\forall x\in A,  \frac{d_A}{4 K_0\log n} \leq d_{\widehat{G}}(x)\leq d_A, \quad \forall y\in B, \frac{d_B}{8}\leq d_{\widehat{G}} (y) \leq d_B$, 
\end{claim}
\begin{proof}

First, we do  a standard dyadic regularization of $G$.  For each $i\geq 1$, let $Y_i=\{y\in Y: 2^{i-1} \delta_Y\leq d(y) < 2^{i} \delta_Y\}$.
By the pigeonhole principle,  for some $1\leq i\leq \log n$, the subgraph $G_i$ of $G$ induced by $X\cup Y_i$
has at least $\frac{e(G)}{\log n}$ edges. From $G_i$, we successively
delete vertices in $X$ whose degrees become less than $\frac{\delta_X}{4\log n}$ and vertices in $Y_i$ whose degrees become less
than $2^{i-3}\delta_Y$. Let $\widehat{G}$ denote the remaining subgraph of $G_i$.
Then 
\begin{equation} \label{eq:G*-edge-bound}
e(\widehat{G})\geq \frac{e(G_i)}{2}\geq \frac{e(G)}{2\log n}.
\end{equation}
Let $A=V(\widehat{G})\cap X, B=V(\widehat{G})\cap Y$.
Let $d_A= K_0 \delta_X$ and $d_B=2^{i}\delta_Y$.  Then
\begin{equation} \label{eq:G*-conditions}
 \forall x\in A,  \frac{d_A}{4 K_0\log n} \leq d_{\widehat{G}}(x)\leq d_A, \quad \forall y\in B, \frac{d_B}{8}\leq d_{\widehat{G}} (y) \leq d_B.
 \end{equation}

From this point on we work with $\widehat{G}$ and will drop the subscripts whenever the context is clear. By a similar proof as in  Lemma \ref{lem:2-paths}, by choosing $C_{4}$ to be sufficiently large,
we may assume that the number of $L$-heavy $2$-paths in $\widehat{G}$ centered at each $y\in B$ is at most $\frac{1}{2^8} d_B^2$. Hence the number of heavy $2$-paths centered in $B$ is at most
$|B| \frac{1}{2^8} d_B^2$. Let $\Pc$ denote the family of $L$-light $2$-paths in $\widehat{G}$ that are centered in $B$. Then
\begin{equation} \label{eq:P-lower-2}
|\Pc|\geq |B|\frac{d_B}{8}\left(\frac{d_B}{8}-1\right)- \frac{1}{2^8} |B|d_B^2 \geq  \frac{1}{2^8}d_B e(\widehat{G}) =\frac{1}{2^8}d_B \sum_{x\in A} d_{\widehat{G}}(x).
\end{equation}

\end{proof}

We clean $\Pc$ to get a subfamily $\Pc'$ by removing members as follows. Whenever there is a member $P\in \Pc$
whose first vertex $x$ is the first vertex of fewer than $\frac{1}{2^{9}} d_B d_{\widehat{G}}(x)$ many members of
$\Pc$, remove all such members. By~\eqref{eq:P-lower-2}, $|\Pc'| \geq \frac{1}{2} |\Pc|$. 

We now seek to apply Lemma~\ref{lem:spiders-to-graph} to this collection $\Pc'$. Observe that for all $x \in A$, $d_{\widehat{G}}(x) \geq \frac{d_A}{4 K_0\log n}$, so we will let $M = \frac{1}{2^{11} K_0 \log n} d_A d_B$. Since clearly each initial segment $wb$ of a member of $\Pc'$ extends to at most $d_B$ members of $\Pc'$,
we will let $\lambda = \left( \frac{1}{2^{11} K_0 \log(n)} d_A \right)^{-1}$.
By ~\eqref{eq:P-lower-2}, $|\Pc|\geq \frac{1}{2^{10} K_0\log n}|A| d_Bd_A$. This together with
assumption three of the lemma gives  
\[|\Pc'| M^{q - 1} \geq \left(\frac{1}{2^{11} K_0\log n}\right)^q |A|d_A^qd_B^q\geq \left(\frac{1}{2^{11}}\right)^q C_4 |A||X|^{q-1}\geq  C_{3} |A|^q,\] by choosing $C_4$ to be sufficiently large in terms of $C_3, q$.
So condition three of Lemma~\ref{lem:spiders-to-graph} holds. Similarly, by assumption four of the lemma, we have that 
\[\lambda^{-1} M \geq \left(\frac{1}{2^{11}K_0\log n}\right)^2 d_A^2d_B\geq \left(\frac{1}{2^{11}\log n}\right)^2\delta_X^2\delta_Y\geq 
\frac{1}{2^{22}} C_4 |X| \geq C_{3} |A|,\] 
by choosing $C_4$ to be sufficiently large in terms of $C_3, q$.
So condition four of Lemma~\ref{lem:spiders-to-graph} holds. Thus, $\widehat{G}$ contains a copy of $H'$, and thus $G$ does.  
\end{proof}

We close this section with a quick corollary of the tools we developed in this section.
Though the result is not directly used in our main proof, it might be of independent interests.


\begin{corollary} \label{cor:subdiv}
Let $q\geq 3$ be an integer.
Let $H$ be a bipartite graph with a bipartition $(V_1,V_2)$ such
that its neighborhood hypergraph $\N_{V_1}(H)$ induced on $V_1$ is $q$-partite.
Let $H'$ be the subdivision of $H$.
Then $\ex(n,H')\leq O(n^{\frac{3}{2}-\frac{1}{2q}})$.
\end{corollary}
\begin{proof}
Let $\ve = \frac{1}{8}$, and $L \geq  \max \{(s^2 + 2s)^{1/4}, v(H')\}$. Let $C_{3}$ be as in Lemma~\ref{lem:spiders-to-graph} applied with $H$ and $L$. Let $G$ be any $n$-vertex graph with $e(G)\geq Cn^{\frac{3}{2}-\frac{1}{2q}}$, where $C$ is a sufficiently large constant to be specified. We apply Lemma~\ref{lem:almost-regular} to find a subgraph of $G'$ on some $m$ vertices  such that for all $v\in V(G')$, $d\leq d(v)\leq Kd$ and $d \geq \frac{2}{5 K} C m^{\frac{1}{2} - \frac{1}{2 q}}$. 

By Lemma~\ref{lem:2-paths}, we have that the number of $L$-heavy $2$-paths is at most $\ve m d^2$. Let $\Pc$ be the family of non-$L$-heavy $2$-paths in $G'$. Note that as the number of $2$-paths in $G'$ is at least $\frac{1}{4} md^2$, we have that $|\Pc| \geq \frac{1}{8} m d^2$. For each vertex $v \in V(G')$, if it is the first vertex for fewer $\frac{1}{16}d^2$ members of $\Pc$, remove all paths starting at $v$ from $\Pc$. Call the resulting collection $\Pc'$ and note it is of size at least $\frac{1}{16} m d^2$. 

We seek now to apply Lemma~\ref{lem:spiders-to-graph} to $\Pc'$. By our cleaning, we can take $M = \frac{1}{16}d^2$.  By our maximum degree condition on $G'$, we can take $\lambda = 16 K d^{-1}$. Since $d \geq  \frac{2}{5 K} C m^{\frac{1}{2} - \frac{1}{2 q}}$ and $C$ is sufficiently large in terms of $C_{3}$, we have that $|\Pc'| M^{q - 1} \geq C_{3} m^q$. Note further that as $\lambda^{-1} M \geq \frac{1}{2^{7}K} d^3 \geq C_{3} m$, for $q \geq 3$ and $C$ sufficiently large in terms of $C_{3}$. Thus, we may apply Lemma~\ref{lem:spiders-to-graph} to $\Pc'$ to find a copy of $H'$ inside $G$. 

\end{proof}
Note that it is well-known  \cite{AKS, Furedi} that for any bipartite graph $H$, $\ex(n,H')=O(n^{\frac{3}{2}})$, while
by the work of Janzer \cite{Janzer-bipartite} it was known that $\ex(n,K'_{q,t})=O(n^{\frac{3}{2}-\frac{1}{2q}})$.
For $q\geq 3$, Corollary \ref{cor:subdiv} generalizes the result for $K'_{q,t}$.


\section{Proof of Theorem~\ref{thm:main} }  \label{sec:main-proof}
\stepcounter{theorem}

\begin{proof}[Proof of Theorem~\ref{thm:main}]

 Let $K$ be as in Lemma~\ref{lem:almost-regular}. Let $L_1$ be the constant returned by Lemma~\ref{lem:short-spiders} for our chioces of $r, t, \ell$, and $L_2$ be the constant returned by Lemma~\ref{lem:spider} for our choices of $r, t, \ell$. Let $C_{0}$ be the maximum constant $C$ and $L_3$ the maximum constant $L_3$ returned by Lemma~\ref{lem:strong-subtree-lemma} over possible inputs of $D \subseteq T_{r, t}'$ and our fixed choices of $r, t, \ell$. Let $L = \max\left\{\left(3\ell(1 + r + rt)\right)^{1/2}, L_1, L_2, L_3, v(H_{r, t}^{\ell})\right\}$. Let $C_{1}$ be the constant returned by Lemma~\ref{lem:strong-spiders} applied with $L$ and $H$.  Let $C_{2}$ be the constant returned by Lemma~\ref{lem:subdiv-lem1} applied with the graph $H_{r, t}^{\ell}$. Let $C_{3}$ be the constant returned by Lemma~\ref{lem:spiders-to-graph} applied with $H = F_{r, t}^{\ell}$ and $L$.

We use the notation $c\ll f(a,b)$  (or analogously, $c\gg f(a,b)$   to mean that given  a choice of $a,b$, $c$ is chosen to be sufficiently small (large) in terms of $a$, $b$. Given $r, t, \ell$, subsequently fix a choice of constants respecting  dependencies as stated below: 
\begin{itemize}
\item  $\alpha\ll f_1(r,t,\ell,L)$,  $\lambda \ll f_2(r,t,\ell, K)$, $\eta \ll f_3(\alpha, r, t)$, $\eta_1, \eta_2  \ll f_4(r,t, L,  K, \eta, C_{2})$ , $\gamma \ll f_5( r, t, L,K, \eta_1)$ 
\item   $C_{5}\gg f_6(L, r, t,  C_{1}, C_{2}, C_{4})$, $C_{6} \gg f_7(  t, \alpha, \gamma, C_{2}, C_{3})$, $C_{7}\gg f_6( r, t, L,  \alpha, \lambda)$,
\item $C_{8}\gg f_8(r, t, L, K, \alpha, \eta, \eta_1, \eta_2, C_{2}, C_{4},  C_{5},  C_{6}, C_{7})$, 
\end{itemize}
where  in the last lines $C_{4}$ is the constant returned by Lemma~\ref{lem:subdiv-lem3}, with $K_0 = 4 \eta^{-1}_1 K$, $q = t + 1,$ and $H = F_{r, t}^\ell$.

Let $C_{9} = 5 K C_{8}$ and $G$ be a graph on $N$ vertices with no copy of $H_{r, t}^\ell$ for $N$ sufficiently large. We seek to show that $e(G) < C_{9}N^{1 + \frac{rt - 1}{2(rt + r)}}$. 
Suppose for contradiction that  $e(G) \geq  C_{9}N^{1 + \frac{rt - 1}{2(rt + r)}}$. 

We start by applying Lemma~\ref{lem:almost-regular} to find a bipartite subgraph $G'$ of $G$ on some $n$ vertices 
such that for all $v\in V(G')$, 
\[ d\leq d(v)\leq Kd,\]
where $d\geq C_{8} n^{\frac{rt-1}{2rt+2r}}$ and $C_{8} = \frac{ C_{9}}{5K}$.
(Lemma~\ref{lem:almost-regular} doesn't necessarily return a bipartite subgraph, but
at a cost of a factor of $1/2$, one may obtain a bipartite one.)
By taking $N$ sufficiently large, we can make $n$ and $d$ sufficiently large.

Throughout the proof, we use the following notation. Given any family $\G$ of labeled copies
of $T'_{r,t}$ and a vector $U$ of at most $rt$ vertices in $G'$, we let $\G_U$ the subfamily of
members of $\G$ that map the first $|U|$ leaves of $T'_{r,t}$ to $U$. So, in particular,
for any $Z\in [V(G')]_{rt}$, $\G_Z$ is the subfamily of members of $\G$ that have leaf vector $Z$.

Let $\F$ be the family of all labeled copies of $T_{r, t}'$ in $G'$. For $n$ sufficiently large,
by a simple greedy process, one can show that
\begin{equation}\label{eq:F-lower}
|\F|\geq  n \left( \frac{d}{2} \right)^{2rt+2r}.
\end{equation}

By Lemma~\ref{lem:2-paths}, Lemma~\ref{lem:short-spiders}, and Lemma~\ref{lem:spider} with $\ve = \frac{1}{(32K)^{2(r + rt)}}$, 
for $d$ sufficiently large,
the number of members of $\F$ that contains an $L$-heavy $2$-path, or an $L$-heavy copy of
$S_{t+1}^-$, or an $L$-heavy copy of $S_{t+1}$ is at most
\[ \ve [2^{2rt + 2r} nd^2 (Kd)^{2rt+r-2} + 2^{2rt + 2r} n d^{2r-1} (Kd)^{2rt+1} + 2^{2rt + 2r} n d^{2r}(Kd)^{2rt}]\]


Let $\F^{(1)}$ be obtained from $\F$ be removing all members that contain an $L$-heavy $2$-path,
or an $L$-heavy $S_{t+1}^-$ or an $L$-heavy $S_{t+1}$. For $d$ sufficiently large, we have
\begin{equation}\label{eq:F1-lower}
|\F^{(1)}|\geq  \frac{1}{2}n \left( \frac{d}{2} \right)^{2rt+2r}.
\end{equation}

Let $\cH$ be obtained from $\F^{(1)}$ by removing
from $\F^{(1)}$
subfamilies $\F^{(1)}_Z$ with $|\F^{(1)}_Z|\leq \frac{1}{4}|\F^{(1)}| n^{-rt} $.

By our cleaning process, we have
\begin{equation} \label{eq:H-lower}
|\cH|\geq \frac{1}{2}|\F^{(1)}| 
\end{equation}
\begin{equation} \label{eq:H-lower-part2}
\forall Z\in [V(G')]_{rt}  \mbox{ if }  \cH_Z\neq \emptyset  \mbox{ then } 
|\cH_Z|\geq \frac{1}{4} |\F^{(1)}|n^{-rt}\geq \frac{1}{2} |\cH|n^{-rt}, 
\end{equation}

Next, we define types of members of $\cH$ and define subfamilies of $\cH$ accordingly.
Recall that $T'_{r.t}$ has edge set $E=\{wx_i: i\in [r]\}\cup \{x_iy_i: i\in [i]\}\cup\{y_iv_{ij}: i\in [r], j\in [t]\}\cup \{v_{i,j}z_{i,j}: i\in [r], j\in [t]\}$. 
Consider any $Z\in [V(G')]_{rt}$ such that $\cH_Z\neq \emptyset$.
If $\cH_Z$ has an $(\alpha,x_i)$-anchor $v$ for some $i\in [r]$, we fix such $i$ and $v$,  let $c(Z)=i$, $a(Z)=v$,
and let $\A_Z$ be the subfamily of members of $\cH_Z$ that map $x_i$ to $v$.
We say that $\cH_Z$ is of type 1. By definition, $|\A_Z|\geq \alpha|\cH_Z|$.
If $\cH_Z$ is not of type 1, but  has an $(\alpha,y_i)$-anchor $v$ for some $i\in [r]$, we
fix such $i,v$, let $c(Z)=i$, $a(Z)=v$, and let $\A_Z$ be the subfamily of $\cH_Z$ that map
$y_i$ to $v$. By definition, $|\A_Z|\geq \alpha|\cH_Z|$. We say that $\cH_Z$ is of type 2.
If $\cH_Z$ is neither of type 1 nor of type 2, we say that $\cH_Z$ is of type 3. Define

\begin{eqnarray*}
\cH^{(1)}&=&\bigcup \left\{\cH_Z: \cH_Z \mbox{ is of type 1} \right\},\\
\cH^{(2)}&=&\bigcup \left\{\cH_Z: \cH_Z \mbox{ is of type 2}  \right\},\\
\cH^{(3)}&=&\bigcup \left \{\cH_Z: \cH_Z  \mbox{ is of type 3} \right \}.
\end{eqnarray*}

To prove the theorem, we need to consider two main cases. 
The first case is the most technical case and will involve subscases.
Throughout the rest of the discussion, for ease of discussion, we will occasionally abuse notation and use the same letter, say $U$,
to denote both a vector of vertices and the underlying unordered set of vertices. 
We say a vector $U$ of vertices is disjoint from another vector $V$ of vertices, if the two unordered sets are disjoint.
\medskip

\subsection{The case where $|\cH^{(1)}|\geq \frac{1}{3}|\cH|$ or $|\cH^{(2)}|\geq \frac{1}{3} |\cH|$}
Recall that for each $Z\in [V(G')]_{rt}$ such that $\cH_Z$ is of type 1 or $2$, $c(Z)\in [r]$, $a(Z)$, and
$\A_Z$ are defined. If $|\cH^{(1)}|\geq \frac{1}{3}|\cH|$, let $\G=\cH^{(1)}$. If $|\cH^{(1)}|<\frac{1}{3}|\cH|$
but $|\cH^{(2)}|\geq \frac{1}{3}|\cH|$ instead,
then let $\G=\cH^{(2)}$.
By the pigeonhole principle, there is some $i\in [r]$ such that 
$|\bigcup \{\cH_Z:  \cH_Z\subseteq \G, c(Z)=i\}| \geq \frac{1}{r} |\G|$. Fix such an $i$. Without loss of generality, suppose $i=r$.
Let \[\G^{(1)}=\bigcup \{\A_Z:  \cH_Z\subseteq \G, c(Z)=r\}.\] 
By definition,
$|\G^{(1)}|\geq \frac{\alpha }{r}|\G|\geq \frac{\alpha}{3r}|\cH|$.

 Next, we clean up $\G^{(1)}$ further as follows. Recall the definition of $D$-shadows from Definition ~\ref{defn:D-shadow}.  Recall we may assume $\eta \leq \frac{\alpha}{2^{3 + 4rt + 4r} K^{2rt + 2r} 3r}$.
 Whenever there is a proper subtree $D$ of $T'_{r,t}$ and 
 a member $F\in \G^{(1)}$ such that its $D$-shadow $F[D]$ is the $D$-shadow of
 fewer than $\eta d^{2rt+2r - e(D)}$ members of $\G^{(1)}$, remove all such members.
 The number of members removed is at most $\sum_D n(Kd)^{e(D)}\eta d^{2rt+2r-e(D)}\leq \eta n (\sum_D K^{e(D)})d^{2rt+2r}$, where the sums are over all proper subtrees $D$ of $T'_{r,t}$. Since $\eta$ is chosen to be sufficiently small, this is at most $\frac{1}{2}|\G^{(1)}|.$
 Let $\G^{(2)}$ denote the remaining subfamily of $\G^{(1)}$. 
Then $|\G^{(2)}|\geq \frac{1}{2}|\G^{(1)}|$.

By the pigeonhole principle, there exists $U \in [V(G)]_{(r - 1)t}$ such that $|\G^{(2)}_U| \geq |\G^{(2)}|n^{-(r - 1)t}$. Let us fix such a $U$.
Given two disjoint vectors $U,V$ let $U\vee V$ denote the vector obtained by attaching $V$ to the end of $U$.

For any $V \in [V(G)]_t$ that is disjoint from $U$, we say $V$ is {\it good} if $|\G^{(2)}_{U \vee V}| \geq \frac{1}{2}|\G^{(2)}_U|n^{- t}$. 
Let $\G^* = \bigcup_{V \text{ is good}} \G^{(2)}_{U \vee V}$. Note that $\G^*_U = \G^*$, and to emphasize the constrained structure of $\G^*$, we will refer to it as $\G^*_U$ throughout the rest of this work. Observe $|\G^*_U| \geq \frac{1}{2} |\G^{(2)}_U|$.   Note that by ~\eqref{eq:F1-lower}
and prior discussions,

\begin{equation}\label{eq:Gu-lower}
|\G^{(2)}_U|\geq \frac{\alpha}{24 r} n\left(\frac{d}{2}\right)^{2rt+2r}n^{-(r-1)t}=\frac{\alpha}{24 r}n^{1-(r-1)t} \left(\frac{d}{2}\right)^{2rt+2r}. 
\end{equation}

Let 
\begin{eqnarray*}
W &=& \mbox{ the set of images of $w$ in members of $\G^*_U$} \\
X &=& \mbox{ the set of images of $x_r$ in members of $\G^*_U$}\\
Y &=& \mbox{ the set of images of
$y_r$ in members of $\G^*_U$}.
\end{eqnarray*}
 
\begin{claim} \label{claim:GU-lower}
We have $|\G^{(2)}_U|\geq \max\{\eta |W|d^{2t+2}, \eta |X| d^{2t + 1}, 2C_5n^t\}$ and $|\G_U^*| \geq  \max \{ \frac{\eta}{2}|W|d^{2t + 2}, C_{5}n^t \}$. 
\end{claim}
\begin{proof}
Let $v \in W$. By definition, there exists $F\in \G_U^*$ which maps $w$ to $v$.
Let $D$ be the subtree of $F$ with leaf set  $\{v\}\cup U$. By our cleaning rule
in obtaining $\G^{(2)}$,
$D$ is contained  in at least $\eta d^{2t + 2}$ many members of $\G^{(2)}$, all of which are also members of $\G^{(2)}_U$. Since this holds for each $v\in W$, $|\G^{(2)}_U| \geq \eta |W|d^{2t + 2}$. Hence, $|\G_U^*| \geq \frac{\eta}{2}|W|d^{2t + 2}$. 

Let $x \in X$. There exists $F\in \G_U^*$ that maps to $x_r$ to $x$. 
Let $D$ now be the subtree of $F$ with leaf set  $\{x\}\cup U$. Similar to before, $D$
is contained  in at least $\eta d^{2t + 1}$ many members of $\G^{(2)}$, all of which 
are also in $\G^{(2)}_U$. Since this holds for each $x\in X$,
$|\G^{(2)}_U| \geq \eta |X|d^{2t + 1}$. 

Note that by the above discussion and the assumption $C_8$ is sufficiently large compared to $C_5$, 
\[|\G^{(2)}_U| \geq \frac{\alpha}{24 r} n^{ 1 - (r - 1)t} \left(\frac{d}{2}\right)^{2rt+2r} \geq \frac{\alpha}{24 r} \left(\frac{C_{8}}{2}\right)^{2rt+2r}n^t \geq 2 C_{5} n^t.\]
Thus, $|\G^*_U|\geq C_5n^t$.
\end{proof}
\medskip

\begin{claim} \label{claim:multiplicity}
Let $v,x,y$ be vertices in $G'$. Let $V\in [V(G')]_t$ be disjoint from $U$.
Then 
\begin{enumerate}
\item There are at most $(Kd)^{2t+2}L^{4(t+1)^2(r-1)}$ members of $\G^*_U$ that map $w$ to $v$.
\item There are at most $(Kd)^{2t+1}L^{4(t+1)^2 (r-1)}$ members of $\G^*_U$
that map $w$ to $v$ and $x_r$ to $x$. 
\item There are at most $(Kd)^{2t} L^{4(t+1)^2(r-1)}$ members of $\G^*_U$
that map $w$ to $v$, $x_r$ to $x$ and $y_r$ to $y$.
\item There are at most $L^{4(t+1)^2(r-1)+(2t+1)^2}$ members of $\G^*_{U\vee V}$
that map $w$ to $v$, $x_r$ to $x$.
\item There are at most $L^{4(t+1)^2(r-1)+4t}$ members of $\G^*_{U\vee V}$
that map $w$ to $v$, $x_r$ to $x$ and $y_r$ to $y$.
\end{enumerate}
\end{claim}
\begin{proof}
Consider any member $F$ of $\G^*_U$ that map $w$ to $v$. The subtree $D$ of $F$ with leaf set $U\cup \{v\}$ is the union of $r-1$ copies of $S_{t+1}$
Since members of $\G^*$ contain no $L$-heavy $S_{t+1}$, there are at most $L^{4(t+1)^2(r-1)}$ 
choices of $D$. For each fixed $D$, since $G'$ has maximum degree at most $Kd$, there are certainly no more than $(Kd)^{2t+2}$ choices of $F$. Hence, there are at most $(Kd)^{2t+2}L^{4(t+1)^2(r-1)}$ members of $\G^*_U$ that map $w$ to $v$.
For similar reasons, there are at most $(Kd)^{2t+1}L^{4(t+1)^2 (r-1)}$ members of $\G^*_U$ that map $w$ to $v$ 
and $x_r$ to $x$, and  
there are at most $(Kd)^{2t} L^{4(t+1)^2(r-1)}$ members of $\G^*_U$
that map $w$ to $v$, $x_r$ to $x$ and $y_r$ to $y$.

Since members of $\G^*_U$ contain no $L$-heavy copy of $S_{t+1}$, no $L$-heavy copy of $S^-_{t+1}$ and no $L$-heavy $2$-path,
there are at most $L^{4(t+1)^2(r-1)+(2t+1)^2}$ members of $\G^*_{U\vee V}$
that map $w$ to $v$, $x_r$ to $x$ and there are at most $L^{4(t+1)^2(r-1)+4t}$ members of $\G^*_{U\vee V}$
that map $w$ to $v$, $x_r$ to $x$, and $y_r$ to $y$.
\end{proof}

We break now into our two cases. 

\subsubsection{ The subcase where $|\cH^{(1)}|\geq \frac{1}{3} |\cH|$
and $\G=\cH^{(1)}$.}

Let $B$ be the bipartite graph with parts $W$ and $X$ whose edges are
the images of $wx_r$ in the members of $\G^*_U$.
Our goal is to apply Lemma~\ref{lem:subdiv-lem1} to show that $B$ contains $H^\ell_{r,t}$.
By the definition of $X$, for each $x\in X$, there is a good $V$ such that
$x$ is an $(\alpha, x_r)$-anchor of $\G^{(2)}_{U\vee V}=\G^*_{U\vee V}$. First we show that 
\begin{claim}
    $$e(B) \geq \frac{\eta}{2K^{2t + 1}L^{4(t + 1)^2 (r-1) }}|W|d.$$
\end{claim}
\begin{proof}
By Claim~\ref{claim:GU-lower}, $|\G_U^*| \geq \frac{\eta}{2}|W|d^{2t + 2}$. 
On the other hand, for each $vx\in E(B)$ with $v\in W, x\in X$,
by Claim~\ref{claim:multiplicity}, there are at most 
$(Kd)^{2t+1}L^{4(t+1)^2 (r-1)}$ members of $\G^*_U$
that map $w$ to $v$ and $x_r$ to $x$.   Thus, 
  \[e(B)  (Kd)^{2t+1} L^{4(t + 1)^2 (r-1)}  \geq |\G^*_U| \geq \frac{\eta}{2}|W|d^{2t + 2},\] and the bound follows.    
\end{proof}

To finish, we obtain two lower bounds on $\delta_X$. 
\begin{claim}
$$\delta_X \geq \frac{\eta}{ 2L^{4(t + 1)^2(r - 1) + (2t + 1)^2}} |W| d^{2t + 2}n^{-t} \text{ and } \delta_X \geq \frac{C_{5}}{ L^{4(t + 1)^2(r - 1) + (2t + 1)^2}}.$$ 
\end{claim}
\begin{proof}
Now, fix any $x \in X$. By our definition of $X$, there is a good $V$ such that  $x$ is the $(\alpha, x_r)$-anchor
of $\G^*_{U\vee V}$. In other words,
$x$ is the image of $x_r$ for every $F \in \G_{U \vee V}^*$. Since $V$ is good, $|\G_{U \vee V}^*| \geq \frac{1}{2}|\G^{(2)}_U|n^{-t}$. By Claim~\ref{claim:GU-lower}, we have that $|\G_{U \vee V}^*| \geq \frac{\eta}{2} |W| d^{2t + 2}n^{-t}$. For each $vx\in E(B)$, where $v\in W, x\in X$, by Claim~\ref{claim:multiplicity},
there can be at most $L^{4(t + 1)^2(r - 1) + (2t + 1)^2}$ members of $\G^*_{U\vee V}$ that map $w$ to $v$ and $x_r$ to $x$.

Thus, 
\[d_B(x) \cdot L^{4(t + 1)^2(r - 1) + (2t + 1)^2} \geq |\G_{U \vee V}^*| \geq \frac{\eta}{2} |W| d^{2t + 2}n^{-t}.\]
Solving for $d_B(x)$ proves the first bound. 
Also, by Claim~\ref{claim:GU-lower} $|\G^*_U|\geq C_5n^t$. By the goodness of $V$, we have that $|\G^*_{U \vee V}| \geq C_{5}$. Thus, $$d_B(x) \cdot L^{4(t + 1)^2(r - 1) + (2t + 1)^2} \geq |\G_{U \vee V}^*| \geq C_{5},$$
and the second bound follows. 
\end{proof}

Now, we have

\begin{align*} e(B)\cdot \delta_X &\geq \frac{\eta}{2K^{2t + 1}L^{4(t + 1)^2(r - 1)}}|W|d \cdot \frac{\eta}{2 L^{4(t + 1)^2(r - 1) + (2t + 1)^2}} |W| d^{2t + 2}n^{-t}  \\
&\geq C_{2} |W|^2,
\end{align*}
using $d \geq C_{8} n^{ \frac{t}{2t + 3}}$, $rt \geq 2t + 3$ and our assumption $\frac{\eta^2}{4K^{2t + 1}L^{8(t + 1)^2 (r-1) + (2t + 1)^2}} C_{8}^{2t + 3} \geq C_{2}$. Since  $\delta_X \geq \frac{C_{5}}{ L^{4(t + 1)^2(r - 1) + (2t + 1)^2}} \geq C_{2}$, 
we may apply Lemma~\ref{lem:subdiv-lem1} (with $H=H^\ell_{r,t}$) to $B$ with bipartition $(W, X)$, and thus find a copy of $H_{r, t}^\ell$, completing the proof in this case. 
\bigskip

\subsubsection{ The subcase where $|\cH_2|\geq \frac{1}{3} |\cH|$
and $\G=\cH_2$}

In this subcase $\G^*_U$ is defined  relative to $\G=\cH_2$.  We now break the subcase into subsubcases.
For each $x\in X$, let $d(x,W)$ denote the number of vertices $v\in W$ such that some member of $\G^*_U$ 
map $w$ to $v$ and $x_r$ to $x$.

Let $\beta= \frac{1}{r}$. Let
\begin{eqnarray*}
X_1&=&\{x\in X: d(x,W)\leq C_{2}\}\\
X_2&=&\{x\in X: C_{2}<d(x,W)\leq n^{-\beta} |W|n^{-\frac{1}{r}}\}\\
X_3&=&X \setminus (X_1 \cup X_2 )
\end{eqnarray*}

For each $i\in [3]$, let $\J_i$ denote the subfamily of members of $\G^*_U$ that map $x_r$ into $X_i$.
Let $m\in [3]$ be the index $i$ of the largest family among $\J_1,\J_2, \J_3$. 
Then $|\J_m|\geq \frac{1}{3} |\G^*_U|$.

We now define a subfamily $\J^*_m$ of $\J_m$ as follows. 
We call a vector $V\in [V(G')]_t$ that is disjoint from $U$ \textit{good} if $|(\J_m)_{U \vee V}| \geq \frac{1}{6} |\G^*_{U}|n^{-t}$. Let $\J_m^* = \bigcup_{V \text{ good}} (\J_m)_{U \vee V}$. 
Let $X_m^*, Y_m^*$ be the set of images of $x_r, y_r$ over members of $\J_m^*$.
\medskip

{\bf Subcase 2.1} $m=1$. 

\medskip

Let $B$ be the bipartite graph with parts $X_1^*$ and $Y_1^*$ 
consisting of all the images of $x_ry_r$ of members of $\J^*_1$.
Our goal is to apply Lemma~\ref{lem:subdiv-lem3} to $B$ to find a copy of $H^\ell_{r,t}$.
This subcase is quite involved due to the highly unbalanced structure of $B$.
It is for this subcase, then we needed to develop Lemma~\ref{lem:subdiv-lem3}.
First, we give a lower bound on $e(B)$.

\begin{claim}
$$e(B) \geq \eta_1 |X^*_1|d. $$
with $\eta_1 \leq \frac{\eta }{12 C_{2}K^{2t} \cdot L^{4(t + 1)^2(r - 1)} }$.
\end{claim}
\begin{proof}

Let $x \in X_1^*$, then there is an $F \in \J_1^*$ that maps $x_r$ to $x$. Letting $D$ be the subtree of $F$ with leaf set $U\cup\{x\}$, by the definition of $\G^{(2)}_U$, $D$ is contained in at least $\eta d^{2t + 1}$ many members of $\G_U^{(2)}$. Thus, $|\G_U^{(2)}| \geq \eta d^{2t + 1}|X_1^*|$. 
By our assumption on $m$, $|\J_1^*|\geq \frac{1}{12}|\G^{(2)}_U|$. Hence,
$|\J_1^*| \geq \frac{\eta}{12} |X_1^*| d^{2t + 1}$. 
Consider now any $xy\in E(B)$ where $x\in X^*_1, y\in Y^*_1$. Since $X^*_1\subseteq X_1$, there are at most $C_{2}$ different vertices $v\in W$ such that some member of $\G^*_U$ maps $w$ to $v$, $x_r$ to $x$. For each such $v$, by Claim~\ref{claim:multiplicity},
there are at most $(Kd)^{2t} L^{4(t + 1)^2(r - 1)}$ members of $\G^*_U$, and thus of $\J_1^*$, that map $w$ to $v$, $x_r$ to $x$, and $y_r$ to $y$.
Thus, 
\[e(B) \cdot  C_{2} (Kd)^{2t} L^{4(t + 1)^2(r - 1)} \geq 
|\J_1^*|\geq \frac{\eta}{12} |X_1^*|d^{2t+1}.\]
The claim follows.
\end{proof}

Next, we obtain a minimum degree bound for $y \in Y_1^*$. 

\begin{claim}
We have that for every $y \in Y_1^*$, 
$$d_B(y) \geq \max \left\{ \eta_2 |X_1^*| d^{2t + 1}n^{-t}, 4C_{4}\right\},$$

with $
\eta_2 \leq \frac{\eta}{L^{4(t + 1)^2(r - 1) + 4t} C_{2}}$

\end{claim}
\begin{proof}

Consider any $y\in Y_1^*$. By definition, there exists at least one good
$V\in [V(G')]_t$ disjoint from $U$ such that all members of $(\J_1^*)_{U\vee V}$ map $y_r$ to $y$
and $|(\J^*_1)_{U\vee V}|\geq \frac{1}{6}|(\G^*)_U|n^{-t}$.
By Claim~\ref{claim:GU-lower}, we have that $|\G^{(2)}_{U}| \geq \eta |X| d^{2t + 1}$. Thus, we have $|\G^*_U| \geq \frac{\eta}{2} |X| d^{2t + 1}$, and so
 $|(\J^*_1)_{U\vee V}|\geq \frac{\eta}{12} |X_1^*| d^{2t + 1} n^{-t}$. 
Consider any edge $xy\in E(B)$, where $x\in X_1^*, y\in Y^*_1$.
Since $x\in X_1$, there are at most $C_{2}$ vertices $v\in W$ for which
there are members of $\J^*_1$ that map $w,x_r,y_r$ to $v,x,y$, respectively.
For each such $v$, by Claim~\ref{claim:multiplicity}, there are at most $L^{4(t + 1)^2(r - 1) + 4t}$ members of $(\J^*_1)_{U \vee V}$ that map $w,x_r,y_r$ 
to $v,x,y$, respectively.
We thus derive  inequality

$$ L^{4(t + 1)^2(r - 1) + 4t} C_{2} d_B(y) \geq |(\J^*_1)_{U\vee V}|\geq \frac{\eta}{12} |X_1^*| d^{2t + 1} n^{-t}. $$

Similarly, by Claim~\ref{claim:GU-lower}, $|\G^*_U| \geq C_{5} n^t$, we have that $|(\J_1^*)_{U \vee V}| \geq\frac{1}{12} C_{5}$ , and thus we derive the second bound, since $ \frac{C_{5}}{12L^{4(t + 1)^2(r - 1) + 4t}C_{2}} \geq 4 C_{4}$.  
\end{proof}


We need a lower bound on the size of $|X_1^*|$. Recall $d=\Omega\left(n^{\frac{rt-1}{2rt+2r}}\right)$.

\begin{claim} \label{claim:X1-lower}
$$|X_1^*| \geq\frac{\alpha}{2^{4rt + 9} C_{2}L^{4(t + 1)^2(r - 1)}K^{2t + 1}} n^{1 - (r - 1)t} d^{2rt - 1} \geq 4^{2t + 2} \eta_1^{-(t + 1)}\eta_2^{-(t + 1)} C_{4}n^{\frac{t + 1}{r}} (\log n)^{t + 1} $$ for $t \geq 1, r\geq 2t + 3$. 
\end{claim}
\begin{proof}
By~\eqref{eq:Gu-lower} and the definition of $\J^*_1$,
$|\J_1^*| \geq \frac{\alpha}{2^{5}3^2 r}n^{1-(r-1)t} \left(\frac{d}{2}\right)^{2rt+2r}$. 
On the other hand, $|\J_1^*| \leq  C_{2}|X^*_1| L^{4(t + 1)^2 (r - 1)}(Kd)^{2t + 1}, $ since for each choice of $x \in X_1^*$, there are at most
$C_{2}$ vertices $v$ for which some member of $\J^*_1$ map $w$ to $v$ and $x_r$ to $x$, and for each fixed choice of $v$, by Claim~\ref{claim:multiplicity}, there are at most
 $L^{4(t + 1)^2(r-1)}(Kd)^{2t+1}$ members of $\J^*_1$ that map $w$ to $v$
 and $x_r$ to $x$.
The inequality follows by combining the lower and upper bounds on $|\J_1^*|$.
\end{proof}
Given this graph $B$, we will now pass to a minimum degree subgraph $B'$ so that we can apply Lemma~\ref{lem:subdiv-lem3}. Indeed, we will iteratively remove vertices from $X^*_1$ if their degree drops below $\frac{1}{4} \frac{e(B)}{|X^*_1|}$ and vertices from $Y_1^*$ if their degree drops below $\frac{1}{4} \frac{e(B)}{|Y_1^*|}$.
Then $e(B')\geq \frac{1}{2}e(B)$. Let $\widetilde{X}=X^*_1\cap V(B')$ and $\widetilde{Y}=Y^*_1\cap V(B')$.
Then $\delta_{\widetilde{X}}\geq \frac{1}{4}\eta_1 d$ and $\delta_{\widetilde{Y}}\geq \frac{1}{4}\delta_{Y_1^*}\geq 
\max\{C_{4}, \frac{\eta_2}{4}|X_1^*|d^{2t + 1}n^{- t}\}$, 
We seek to apply Lemma~\ref{lem:subdiv-lem3} to $B'$ with $q = t + 1$, $X = \widetilde{X}$, $Y = \widetilde{Y}$, $\delta_X = \frac{\eta_1}{4} d$, $\delta_Y = \max\{C_{4}, \frac{\eta_2}{4}|X_1^*|d^{2t + 1}n^{- t}\}$, and $K_0 = 4\eta_1^{-1} K$. Also, note that  $d^{2t+2}n^{-t}=n^{-\frac{1}{r}}$. We have
\begin{align*}
\delta_X^{t+1}\delta_Y^{t+1} &\geq 4^{-2t + 2} \eta_1^{t + 1}\eta_2^{t + 1}d^{t + 1}\left(|X_1^*| d^{2t + 1}n^{- t} \right)^{t + 1} \geq 4^{-2t + 2} \eta_1^{t + 1}\eta_2^{t + 1}
d^{t + 1}\left(\frac{|X_1^*|}{dn^{\frac{1}{r}}}\right)^{t + 1}\\&\geq4^{-2t + 2} \eta_1^{t + 1}\eta_2^{t + 1} |X_1^*|^{t} \frac{|X_1^*|}{n^{\frac{t + 1}{r}}}
\geq C_{4} |X_1^*|^{t} (\log n)^{t + 1},
\end{align*}
where in the last inequality we apply our lower bound on $|X_1^*|$ from Claim~\ref{claim:X1-lower}.

Since $rt > 2t + 3$, we have that $d^{2t + 3} \geq 2^6\eta_1^{-2} \eta_2^{-1} C_{4} n^t (\log n )^{2}$ by choice of $C_{8}^{2t + 3}\geq  2^{6}\eta_1^{-2} \eta_2^{-1} C_{4}$. This implies, by plugging in the values of $\delta_X$ and $\delta_Y$, that 
\[\delta_X^2\delta_Y\geq \frac{\eta_1^2 \eta_2}{2^6} d^{2t + 3}n^{-t} |X_1^*| \geq  C_{4} |X_1^*| (\log{n} )^{2}.\]
Thus, all the conditions of Lemma~\ref{lem:subdiv-lem3} hold, and so $H_{r, t}^\ell \subseteq B$, contradicting $G$ being $H_{r, t}^{\ell}$-free. 

\bigskip

{\bf Subcase 2.2} $m=2$.

Let $\cP$ denote the set of images of $wx_ry_r$ over all members of $\J^*_2$. Note that every member of $\cP$ is an $L$-light $2$-path in $G$ by the definition of $\J^*_2 \subseteq \F^{(1)}$. 

\begin{claim} \label{claim:P-lower}
$|\cP|\geq \gamma  d^2  |W|$ with $\gamma = \frac{\eta}{12K^{2t} L^{4(t + 1)^2(r - 1)}}$.
\end{claim}
\begin{proof}
By Claim~\ref{claim:GU-lower}, $|\G^{(2)}_U| \geq \eta |W| d^{2t + 2}$. By our assumption on $\J_2$ and the relative sizes of $\G^*_U$ and $\J^*_2$, we have that $|\J^*_2| \geq \frac{\eta}{12} |W| d^{2t + 2}$. On the other hand, by  Claim~\ref{claim:multiplicity}, each member of $\cP$ is the image
of $wx_ry_r$ of at most $(Kd)^{2t} L^{4(t + 1)^2(r - 1)}$ many members of $\J^*_2$. Thus, we have 
$$|\cP| (Kd)^{2t} L^{4(t + 1)^2(r - 1)} \geq \frac{\eta}{12} |W| d^{2t + 2},$$

and so the bound follows. 
\end{proof}

For each $y \in Y^*_2$, let $\cP_y$ denote the subfamily members of of $\cP$ that maps $y_r$ to $y$. Recall that $d \geq C_{8} n^{\frac{rt-1}{2rt+2r}}$.
\begin{claim} \label{claim:py-lower}
For each $y \in Y^*_2$, $|\cP_y|\geq C_{6} |W|n^{-\frac{1}{r}}$, where $C_{6} =  \frac{ \eta C_{8}^{2t + 2}}{12 L^{4(t + 1)^2(r - 1) + 4t} }$. 
\end{claim}
\begin{proof}
Let $y \in Y^*_2$. Then, by definition there exists a good $V \in [V(G)]^t$ such that
$|(\J_2^*)_{U \vee V}| \geq \frac{1}{12}|\G^{(2)}_U| n^{- t}$ and 
$y$ is the image of $y_r$ for every $F \in (\J_2^*)_{U \vee V}$. Since $|\G^{(2)}_U| \geq \eta |W| d^{2t + 2}$, we have that $|(\J_2^*)_{U \vee V}|  \geq \frac{\eta}{12} |W| d^{2t + 2}n^{-t}$. 

On the other hand by Claim~\ref{claim:multiplicity}, the number of members of $(\J^*_2)_{U\vee V}$ that
map $wx_ry_r$ to $vxy$ is at most $L^{4(t + 1)^2(r - 1) + 4t}$. Thus, the inequality 
\[|\cP_y|L^{4(t + 1)^2(r - 1) + 4t} \geq \frac{\eta}{12} |W| d^{2t + 2} n^{-t}\] holds.
Solving this for $|\cP_y|$ and using $d=C_{8}n^{\frac{rt-1}{2rt+2r}}$, we find $|\cP_y|\geq C_{6} |W|n^{-\frac{1}{r}}$,
using that $C_8$ is sufficiently large compared to $C_6$.
\end{proof}

We seek now to apply Lemma~\ref{lem:spiders-to-graph} to $\Pc$ to find a copy of $H_{r, t}^{\ell}$. 
By Claim~\ref{claim:py-lower}, we have that we can take $M = C_{6} |W| n^{- \frac{1}{r}}$. Combining this with Claim~\ref{claim:P-lower}, we have that $|\cP| M^{t} \geq \gamma C_{6}^t d^2 |W|^{t + 1} n^{-\frac{  t}{r}}$. This will be greater than $C_{3} |W|^{t + 1}$ by our assumption $C_{6}$ is sufficiently large in terms of $C_{3}$ and that $d^2 \geq n^{t / r}$ for $r \geq t + 2$. Since members of $\Pc$ are $2$-paths contained in members of $\J_2^*$,
by definition of $\J^*_2$ if we fix the first two vertices of any member of $\Pc$, this initial segment extends 
to at most $|W|n^{- \frac{1}{r} -\beta}$ members of $\Pc$.  Thus, we can take $\lambda = \frac{1}{C_{6}} n^{- \beta}$. 

Therefore, as  $ \lambda^{-1} M \geq C_6^2 n^{\beta - \frac{1}{r}} |W| \geq C_{3} |W|$ for $\beta \geq \frac{1}{r}$ and $C_{6}$ sufficiently large in terms of $C_{3}$, we thus may apply Lemma~\ref{lem:spiders-to-graph} 
(with $H=F_{r,t}^\ell$) to find our copy of $H_{r, t}^{\ell}$, as desired. 
\bigskip

{\bf Subcase 2.3} $m=3$.

\medskip
Let $B$ be the bipartite graph with parts $W$ and $X^*_3$ whose edges are images of $wx_r$ over members of $\J^*_3$.

\begin{claim}
$$e(B) \geq \frac{\eta}{3  L^{4(t+1)^2 (r-1) + (2t+1)^2} K^{2t + 1}} |W| d.$$
\end{claim}
\begin{proof} We have $|\J^*_3|\geq \frac{1}{3}|\G_U^*| \geq \frac{\eta}{3} |W| d^{2t + 2}$. Consider any edge
$vx\in E(B)$, where $v\in W$ and $x\in X^*_3$.
By Claim~\ref{claim:multiplicity}, there are no more than $(Kd)^{2t+1}L^{4(t+1)^2 (r-1)+(2t+1)^2}$ many members of $\G_U^*\supseteq \J^*_3$ that map $w$ to $v$ and $x_r$ to $x$. Hence, 

$$e(B) \cdot (Kd)^{2t+1}L^{4(t+1)^2 (r-1)+(2t+1)^2} \geq \frac{1}{3}  \eta |W| d^{2t + 2}, $$
and so the claim follows. 
\end{proof}

By definition, $\delta_{X^*_3} \geq n^{-\beta- \frac{1}{r}}|W|$. Furthermore, as $X^{^*}_3$ is nonempty, $\delta_{X^*_3} \geq C_{2}$.  Thus, we have

\[e(B)\cdot \delta_{X_3^*} \geq \frac{\eta}{3  L^{4(t+1)^2 (r-1) + (2t+1)^2} K^{2t + 1}} dn^{-\beta - \frac{1}{r}} |W|^2 \geq C_{2}|W|^2, \]
as long as $d \geq C_{2} \eta^{-1} L^{4(t+1)^2 (r-1) + (2t+1)^2} K^{2t + 1} n^{\beta +  \frac{1}{r}}$, which holds as long as $C_{8}$ is sufficiently large and $\frac{rt-1}{2rt+2r}\geq \frac{2}{r}$, which is guaranteed if
$rt \geq 4t + 5$, which is guaranteed by our assumptions of the theorem.
By Lemma~\ref{lem:subdiv-lem1}, $B$ contains $H_{r, t}^{\ell}$, contradicting $G$ being $H_{r, t}^{\ell}$-free. 


\vskip 60pt

\subsection{ The case where $|\cH_3|\geq \frac{1}{3} |\cH|$}

\stepcounter{theorem}
\bigskip

We pass from $\cH_3$ to $\cH^*$ by letting $\cH^*$ be the subfamily of $\cH_3$ defined by the union over $Z \in [V(G)]_{rt}$ where $|(\cH_3)_Z| \geq \frac{1}{6} |\cH|n^{-rt} \geq C_{7}$ for $C_{8}$ sufficiently large in terms of $C_{7}$. Note that by ~\eqref{eq:F1-lower},
\begin{equation}\label{eq:H*-lower}
|\cH^*| \geq \frac{1}{6}|\cH| \geq \frac{1}{12}|\F^{(1)}|\geq\frac{n}{24} \left( \frac{d}{2} \right)^{2r + 2rt}.
\end{equation}
Recall the definition of $D$-shadows from Definition~\ref{defn:D-shadow} and strong family from 
Definition~\ref{def:strong}.

We say a proper subtree $D \subseteq F'_{r, t}$ is \textit{nice} if $D$ is obtained from
$F'_{r,t}$ by deleting at most one $z_{i,j}$ for each $i\in [r]$. In particular, $D$ has $rt$ leaves. Recall that $\lambda = \frac{1}{(rt + 1 + 2r)^2\ell}$.

\begin{claim}\label{claim:iterative}
Let $D$ be a nice subtree of $T'_{r,t}$ with $2rt+2r-p$ edges, where $0\leq p\leq r-1$.
Let $U\in [V(G')]_{rt}$.
Suppose  $\D$ is a subfamily of members of $\partial_D(\cH_Z^*)$ that have
leaf vector $U$ and  $|\D| \geq \lambda^{p} |\cH^*_Z|$.
Then, either
\begin{enumerate}
\item The family $\D$ is $L$-strong, or 
\item There exists a nice subtree $D'$ of $T'_{r,t}$ contained in $D$
 with $e(D') = e(D) - 1$, $U' \in [V(G)]_{rt}$, and a subfamily $\D'$ of members of
 $\partial_{D'}(\cH^*_Z)$ such that 
 every member of $\D'$ has leaf vector $U'$ and $|\D'| \geq \lambda|\D| $.
\end{enumerate}

\end{claim}
\begin{proof}
If the family $\D$ is $L$-strong, we are done. So suppose on the other hand that the family $\D$ is not $L$-strong. Then by Lemma~\ref{lem:anchored}, there exists a vertex $u \in V(D)$ and a vertex $v \in V(G)$ such that $v$ is a $(\lambda, u)$-anchor for $\D$. Let $\A$ denote the subfamily of members of $\D$ that map $u$ to $v$.
Then $|\A|\geq \lambda |\D|\geq \lambda^{2rt+2r+1-e(D)}|\cH^*_Z|$.

If $u$=$x_i$ or $y_i$ for some $i\in [r]$, then it would follow that the family $\cH^*_Z$ has an $(\alpha, x_i)$ or $(\alpha, y_i)$-anchor respectively, contradicting members of $\cH^*_Z$ being of type 3. Indeed, if for example $u$ were a $(\lambda, x_i)$-anchor for $\D$, then it would be a $(\lambda^{2rt+2r - e(D) + 1}, x_i)$-anchor for $\cH^*_Z$. 

Suppose next $u = w$. Then, since $\cH^*_Z$ contains no $L$-heavy $S_{t+1}$ or $S_{t + 1}^-$, we have that $|\A| \leq  L^{4(t + 1)^2r} < \lambda^{2rt+2r+1-e(D)}|\cH^*_Z|$,  as $|\cH^*_Z| \geq C_{7}$ is sufficiently large in terms of $L, \lambda$. Hence, $u\in \bigcup_{i=1}^r\bigcup_{j=1}^t \{v_{i,j}\}$.

Without loss of generality, suppose $D=T'_{r,t}-\{z_{1,1},z_{2,1},\dots, z_{p,1}\}$. Suppose first that
$u\in \bigcup_{i=p+1}^r \bigcup_{j=1}^t \{v_{i,j}\}$. Without loss of generality, suppose $u=v_{p+1,1}$.
Let $D'=D-\{z_{p+1,1}\}$. Let $\D' = \partial_{D'}(\A)$ and $U'$ being common leaf vector of members of $\D'$.
Then the claim holds. Hence, we may assume that $u\in \bigcup_{i=1}^p \bigcup_{j=2}^t \{v_{i,j}\}$. Without loss of generality,
suppose $u=v_{1,2}$. 

Since members of $\A$ contain no $2$-paths that are $L$-heavy in $G'$, there are at most $L^4$ different images of $y_1$ over all
members of $\A$. Hence at least $L^{-4}|\A|$ members of $\A$ map $y_1$ to the same vertex $y$.
Thus $y$ is an $( \lambda^{2rt+2r - e(D) + 1}L^{-4}, y_1)$-anchor, and thus an $(\alpha, y_1)$-anchor for $\cH^*_Z$, contradicting members of $\cH^*_Z$ being of type 3. 
\end{proof}

\begin{claim} \label{claim:strong-subtree-containment}
For every $Z\in [V(G')]_{rt}$ with $\cH^*_Z \neq \emptyset$, there exists a tuple $U \in [V(G)]_{rt}$ and a nice subtree $D$ of $T'_{r,t}$ such that
the subfamily $\D$ of members of $\partial_D(\cH_Z^*)$ that have leaf vector $U$ is $L$-strong and $|\D| \geq \lambda^r |\cH^*_Z|$.  In particular, at least $\lambda^r |\cH^*_Z|$ members of $\cH^*_Z$ contain an $L$-strong nice subtree.
 \end{claim}
 \begin{proof}

 We  will iteratively apply Claim~\ref{claim:iterative} to find the tuple $U \in [V(G)]_{rt}$ and nice subtree $D$ of $T'_{r,t}$ such that the subfamily $\D$ of members of $\partial_D(\cH_Z^*)$ that have leaf vector $U$ is $L$-strong.
Let $U_0 = Z$ and $D_0 = T_{r, t}'$, and $\D_0 = \cH_Z^*$. Observe that the family $\D_0$ satisfies the conditions of Claim~\ref{claim:iterative} with $D = D_0$ and $U = U_0$. In general, let $0\leq j\leq r-1$. Suppose $D_j$ is a nice subtree of
$T'_{r,t}$ with $2rt+2r-j$ edges and $\D_j$ is a subfamily of members of $\partial_D(\cH_Z^*)$ that have the same
leaf vector $U_j$ and  $|\D_j| \geq \lambda^{j} |\cH^*_Z|$. If $\D_j$ is $L$-strong, then the lemma holds. Otherwise, 
we apply Claim~\ref{claim:iterative} to obtain a nice subtree $D_{j+1}$ with $2rt+2r-j-1$ edges and a subfamily $\D_{j+1}$
of $\partial_{D_{j+1}}(\cH^*_Z)$ with the same leaf vector $U_{j+1}$ such that $|\D_{j+1}|\geq \lambda^{j+1}|\cH^*_Z|$.
Since every nice subtree of $T_{r, t}'$ has at least $2rt+2r- r$ edges, the process must terminate with a nice subtree $D_m$
of $2rt+2r-m$ edges and a subfamily $\D_m$ of members of $\partial_{D_m}(\cH^*_Z)$ that have the same leaf vector and has size
at least $\lambda^m|\cH^*_Z|$ for some $m\in [r]$.
\end{proof}

By Claim~\ref{claim:strong-subtree-containment}, for $\lambda^r$ proportion of members $F$ of $\cH^*$,
there is a nice subtree $D$ of $T'_{r,t}$ such that $F[D]$ is $L$-strong in $G'$. 
On the other hand, for each nice subtree $D$ of $T'_{r,t}$,
by Lemma~\ref{lem:strong-subtree-lemma} and the fact that $\Delta(G')\leq Kd$, 
there can be at most $2C_0 nd^{e(D)-1}\cdot (Kd)^{2rt+2r-e(D)}<2C_0K^{2rt+2r} nd^{2rt+2r-1}$
members of $\cH^*$ such that $F[D]$ is $L$-strong. Since there are at most $(t + 1)^r$ choices of $D$,
we have for sufficiently large $n$,
\[|\cH^*| < 2(t + 1)^r \lambda^{-r} C_0K^{2rt+2r} nd^{2rt+2r-1} < \frac{n}{24} \left( \frac{d}{2}\right)^{2rt + 2r},\]
contradicting ~\eqref{eq:H*-lower}. This contradiction completes our proof of Theorem~\ref{thm:main}. 
\end{proof}

\section{Concluding remarks}

\begin{enumerate}
\item Besides the Bukh-Conlon conjecture, another intriguing conjecture posed by Kang, Kim, Liu \cite{KKL} states that if $H$ is a bipartite graph satisfying that $\ex(n,H)=O(n^{1+\alpha})$ then $\ex(n,H')=O(n^{1+\frac{\alpha}{2}})$. As shown by the authors, the conjecture, if true, also implies the rational exponents conjecture. Besides implying the rational exponents conjecture, the conjecture is very interesting on its own. We note that our main theorem
proves the conjecture for the family of rooted powers of height two trees $T_{r,t}$ for $t \geq 2, r\geq 2t + 3$.
The conjecture, however, remains wide open in general. 

\item The {\it $p$-subdivision} of a graph $H$ is obtained from $H$ by replacing edges
of $H$ with internally disjoint paths of length $p$. So in this language, $H'$ is the $2$-subdivision of $H$.
It is likely our method can be used to establish the Bukh-Conlon conjecture for $p$-subdivisions of
$T_{r,t}$ for any even $p$ when $r$ is moderately large compared to $t$.

\end{enumerate}

\end{document}